\documentclass[11pt]{amsart}
\usepackage{amsmath}
\usepackage{a4wide}
\usepackage[utf8]{inputenc}
\usepackage{amssymb}
\usepackage{amsopn}
\usepackage{epsfig}
\usepackage{amsfonts}
\usepackage{latexsym}
\usepackage{amsthm}
\usepackage{enumerate}
\usepackage[UKenglish]{babel}
\usepackage{verbatim}
\usepackage{color}
\usepackage{pgf}
\usepackage{tikz}

\usepackage{mathrsfs}   

\usetikzlibrary{arrows,automata}

\DeclareMathAlphabet{\pazocal}{OMS}{zplm}{m}{n}
\newtheorem{theorem}{Theorem}[section]
\newtheorem{lemma}[theorem]{Lemma}
\newtheorem{proposition}[theorem]{Proposition}
\newtheorem{corollary}[theorem]{Corollary}
\newtheorem{main}{Theorem}

\theoremstyle{definition}
\newtheorem{definition}[theorem]{Definition}

\theoremstyle{remark}
\newtheorem{remark}[theorem]{Remark}

\numberwithin{equation}{section}

\newcommand{\R}{\ensuremath{\mathbb{R}}}
\newcommand{\N}{\ensuremath{\mathbb{N}}}

\newcommand{\set}[1]{\left\{#1\right\}}
\newcommand{\la}{\lambda}
\newcommand{\ga}{\gamma}

\newcommand{\ep}{\varepsilon}
\newcommand{\f}{\infty}

\newcommand{\al}{\alpha}

\newcommand{\lle}{\preccurlyeq}
\newcommand{\lge}{\succcurlyeq}

\newcommand{\si}{\sigma}

\begin{document}

	\title[Pointwise densities and critical values]{Pointwise densities of homogeneous Cantor measure and critical values}
	
	\author{Derong Kong}
	\address[D. Kong]{College of Mathematics and Statistics, Chongqing University,  401331, Chongqing, P.R.China}
	\email{derongkong@126.com}
	
	\author{Wenxia Li}
\address[W. Li]{Wenxia Li: Department of Mathematics, Shanghai Key Laboratory of PMMP, East China Normal University, Shanghai 200062,
People's Republic of China}
\email{wxli@math.ecnu.edu.cn}
	
	\author{Yuanyuan~Yao}
\address[Y. Yao]{Department of Mathematics, East China University of Science and Technology, Shanghai 200237, P.R. China}
\email{yaoyuanyuan@ecust.edu.cn}



	\begin{abstract}
		Let $N\ge 2$ and $\rho\in(0,1/N^2]$. The homogenous Cantor set $E$ is the self-similar set generated by the iterated function system
	\[
		\set{f_i(x)=\rho x+\frac{i(1-\rho)}{N-1}: i=0,1,\ldots, N-1}.
		\]
 Let  $s=\dim_H E$ be the Hausdorff dimension of $E$, and let $\mu=\mathcal H^s|_E$ be the $s$-dimensional Hausdorff measure restricted to $E$.
		 In this paper  we describe, for each $x\in E$, the pointwise lower $s$-density $\Theta_*^s(\mu,x)$ and
 upper $s$-density $\Theta^{*s}(\mu, x)$ of $\mu$ at  $x$. This extends some early results of Feng et al. (2000). Furthermore, we determine two critical values $a_c$ and $b_c$ for the  sets
		\[
	E_*(a)=\set{x\in E: \Theta_*^s(\mu, x)\ge a}\quad\textrm{and}\quad E^*(b)=\set{x\in E: \Theta^{*s}(\mu, x)\le b}
		\]
	respectively, such that $\dim_H E_*(a)>0$ if and only if $a<a_c$, and that $\dim_H E^*(b)>0$ if and only if $b>b_c$. We emphasize that both values $a_c$ and $b_c$ are related to the Thue-Morse type sequences, and our strategy to find them relies on   ideas from open dynamics and techniques from combinatorics on words.
		
	\end{abstract}

	\maketitle
	
	\section{Introduction}\label{s1}
	Let $N\ge 2$ be an integer, and let $\rho\in(0, 1/N)$. The homogeneous Cantor set $E=E_{N,\rho}$ is the self-similar set generated by the \emph{iterated function system} (IFS)
	\[
	f_i(x)=\rho x+i\; R:=\rho x+i\frac{1-\rho}{N-1},\quad i=0,1,\ldots, N-1.
	\]
	When $N=2$ and $\rho=1/3$, $E_{2, 1/3}$ is the classical middle-third Cantor set. It is easy to see that the convex hull of $E$ is the unit interval $[0, 1]$, and the first level basic intervals $f_{0}([0,1]), f_1([0, 1]), \ldots$ and  $f_{N-1}([0, 1])$ are located one by one from the left to the right (see Figure \ref{fig:1}). These subintervals $f_i([0, 1]), i=0,1,\ldots, N-1$, are pairwise disjoint, and the gaps between any two neighboring subintervals are the same.
	
	 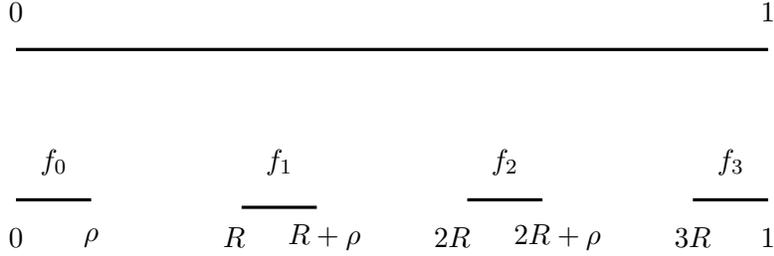
\begin{figure}[h!]
\begin{center}
\begin{tikzpicture}[
    scale=10,
    axis/.style={very thick},
    important line/.style={thick},
    dashed line/.style={dashed, thin},
    pile/.style={thick, ->, >=stealth', shorten <=2pt, shorten
    >=2pt},
    every node/.style={color=black}
    ]
    \draw[axis] (0,0)  -- (1,0) node(xline)[right]{};

      \node[] at (0, 0.05){$0$};

      \node[] at (1, 0.05){$1$};

      \draw[axis] (0,-0.2)  -- (0.1,-0.2) node(xline)[right]{};
 \node[] at (0.05, -0.15){$f_0$};
 \node[] at (0, -0.25){$0$};
 \node[] at (0.1, -0.25){$\rho$};

     \draw[axis] (0.3,-0.21)  -- (0.4,-0.21) node(xline)[right]{};
 \node[] at (0.35, -0.15){$f_1$};
  \node[] at (0.29, -0.25){$R$};
   \node[] at (0.41, -0.25){$R+\rho$};

    \draw[axis] (0.6,-0.2)  -- (0.7,-0.2) node(xline)[right]{};
 \node[] at (0.65, -0.15){$f_{2}$};
  \node[] at (0.58, -0.25){$2R$};
   \node[] at (0.72, -0.25){$2R+\rho$};

        \draw[axis] (0.9,-0.2)  -- (1,-0.2) node(xline)[right]{};
 \node[] at (0.95, -0.15){$f_{3}$};
  \node[] at (0.9, -0.25){$3R$};
   \node[] at (1, -0.25){$1$};

\end{tikzpicture}
\end{center}
\caption{The convex hull $[0, 1]$ of $E$ and the first level basic intervals $f_i([0, 1]), i=0,1,2,3$ with $N=4$ and $\rho=1/10$. Then $R=(1-\rho)/(N-1)=3/10$.}\label{fig:1}
\end{figure}

	Note that for each $x\in E$ there exists a (unique) sequence $(d_i)=d_1d_2\ldots\in\set{0, 1,\ldots, N-1}^\N$ such that
	\begin{equation}\label{eq:proj-map}
	x=\lim_{n\to\f}f_{d_1\ldots d_n}(0)=R\sum_{i=1}^\f d_i\rho^{i-1}=:\pi((d_i)),
	\end{equation}
	where $f_{d_1\ldots d_n}:=f_{d_1}\circ\cdots\circ f_{d_n}$ is the composition of maps. The infinite sequence $(d_i)$ is called a \emph{coding} of $x$. Since $0<\rho< 1/N$, the self-similar IFS $\set{f_i}_{i=0}^{N-1}$ satisfies the strong separation condition (cf.~\cite{Falconer_1990}). So the projection map $\pi: \set{0,1,\ldots, N-1}^\N\to E$ defined in (\ref{eq:proj-map}) is bijective. In other words, each $x\in E$ has a unique coding.
	
	Accordingly, the self-similar set $E$ supports  a unique  measure $\mu=\mu_{N,\rho}$   satisfying
	\begin{equation}\label{eq:mu}
	\mu=\sum_{i=0}^{N-1}\frac{1}{N}\mu\circ f_i^{-1}.
	\end{equation}
The measure $\mu$ is called a \emph{homogeneous Cantor measure}, which is a self-similar measure (cf.~\cite{Hutchinson_1981}).
	In fact, the measure $\mu$ is   the image  measure of the uniform Bernoulli measure on the symbolic space $\set{0,1,\ldots, N-1}^\N$ under the projection map $\pi$. It  is also the $s$-dimensional Hausdorff measure $\mathcal H^s$ restricted to $E$, i.e., $\mu=\mathcal H^s|_E$, where $s=s_{N,\rho}:=-\log N/\log\rho$ is the Hausdorff dimension of $E$. For brevity, we always write $E, \mu$ and $s$ instead of $E_{N,\rho}, \mu_{N,\rho}$ and $s_{N, \rho}$  if no confusion arises.
	
Given  $x\in E$, the \emph{lower and upper $s$-densities} of $\mu$ at $x$ are defined by
	\[
	\Theta_*^s(\mu, x):=\liminf_{r\to 0}\frac{\mu(B(x, r))}{(2r)^s}=\liminf_{r\to 0}\frac{\mathcal H^s(B(x, r)\cap E)}{(2r)^s}
\] and
\[
  \Theta^{*s}(\mu, x):=\limsup_{r\to 0}\frac{\mu(B(x,r))}{(2r)^s}=\limsup_{r\to 0}\frac{\mathcal H^s(B(x, r)\cap E)}{(2r)^s}
	\] respectively,
  where $B(x, r)$ is the open interval $(x-r, x+r)$. The  study of   densities for a self-similar measure attracted a lot of attention  in the literature (see \cite{Ayer-Strichartz-1999, Bedford-Fisher-1992, Falconer-1992, Morters-Preiss-1998, Olsen-2008, Strichartz-Taylor-Zhang-1995} and the references therein).
{When $N=2$ and $\rho\in(0, 1/3]$, Feng et al \cite{Feng-Hua-Wen-2000} explicitly calculated the pointwise densities $\Theta_*^s(\mu, x)$ and $\Theta^{*s}(\mu,x)$ for any $x\in E$. The upper bound $1/3$ for $\rho$ was later  improved to $(\sqrt{3}-1)/2$ by Wang et al.~\cite{Wang-Wu-Xiong-2011}. Motivated by their work, Li and Yao \cite{Li-Yao-2008} determined the pointwise densities of the self-similar measure for non-homogeneous self-similar IFSs.  Dai and Tang \cite{Dai-Tang-2012} considered   the same problem as in \cite{Feng-Hua-Wen-2000} but for $N=3$ and $\rho\in(0, 1/6]$.}

In this paper we consider   the  pointwise  lower and upper   $s$-densities $\Theta_*^s(\mu, x)$ and $\Theta^{*s}(\mu, x)$ for $N\ge 2$ and $\rho\in(0, 1/N^2]$. To state our main results we first define a $N$-to-$1$ map   $T: \bigcup_{i=0}^{N-1}f_i([0,1]) \to [0,1]$ such that
\[
T(x)=f_i^{-1}(x)=\frac{x-iR}{\rho}\quad\textrm{if}\quad x\in f_i([0,1])=[iR, iR+\rho].
\]
Then $T(E)=E$ and $\mu\circ T^{-1}=\mu$. Furthermore, for $x=\pi(d_1d_2\ldots)\in E$ we have $T^n(x)=\pi(d_{n+1}d_{n+1}\ldots)$. So, $\pi$ is the isomorphic map from the symbolic dynamics $(\set{0,1,\ldots, N-1}^\N, \si)$ to the expanding dynamics $(E, T)$, where $\si$ is the left shift map.

Our first result describes the pointwise densities of $\mu$.
	\begin{main}
\label{main:densities}
Let $N\ge 2, 0<\rho\le 1/N^2$ and $s=-\log N/\log \rho$. Suppose $\mu$ is the   self-similar measure supported on $E$ satisfying (\ref{eq:mu}).
\begin{itemize}
\item[{\rm(i)}] For any $x=\pi(d_1d_2\ldots)\in E$ the pointwise  lower $s$-density of $\mu$ at $x$ is given by
\[
\Theta_*^s(\mu, x)=\frac{1}{2^s}\frac{1}{(R/\rho-\liminf_{n\to\f}\ga_n(x))^s},
\]
where
\[
\ga_n(x):=\left\{\begin{array}{lll}
T^n x&\textrm{if}& d_n=0,\\
\max\set{T^n x, 1-T^n x}&\textrm{if}& 1\le d_n\le N-2,\\
1-T^n x&\textrm{if}& d_n=N-1.
\end{array}\right.
\]

\item [{\rm(ii)}] For any $x\in E$ the pointwise upper $s$-density of $\mu$ at $x$ is given by
\[
\Theta^{*s}(\mu, x)= \max\set{\frac{1}{2^s(\liminf_{n\to\f}\eta_n(x))^s}, ~ \limsup_{n\to\f}\frac{1+2\hat d_n}{2^s(\hat d_n R/\rho+\eta_n(x))^s}},
\]
where
\[
\eta_n(x):=\max\set{T^n x, 1-T^n x}\quad\textrm{and}\quad \hat d_n:=\min\set{d_n, N-1-d_n}.
\]

\item[{\rm(iii)}] For $\mu$-almost every $x\in E$ we have
\[
\Theta_*^s(\mu, x)=\left(2R/\rho\right)^{-s} \quad\textrm{and}\quad
\Theta^{*s}(\mu, x)=\left\{\begin{array}{lll}
1&\textrm{if}& N\textrm{ is odd},\\
(N R)^{-s}&\textrm{if}& N\textrm{ is even}.
\end{array}\right.
\]
\end{itemize}
	\end{main}
\begin{remark}\mbox{}

\begin{itemize}
\item
We point out that the upper bound $1/N^2$ for the contraction ratio $\rho$ in Theorem \ref{main:densities} is not optimal. In fact, by a careful estimation one can improve the upper bound to $1/N^{\log3/\log2}$ for the lower density, and to $1/N^{\log 2/\log(3/2)}$ for the upper density.

\item By Theorem \ref{main:densities} (i)  it follows that if $x$ is an endpoint of $E$, i.e., $x$ has a unique coding ending with $0^\f$ or $(N-1)^\f$, then $\ga_n(x)=0$ for all sufficiently large $n$. So,
$
\Theta_*^s(\mu, x)=(2R/\rho)^{-s},
$
which is equal to the typical value of the lower density by Theorem \ref{main:densities} (iii).
 Furthermore, by Theorem \ref{main:densities} (ii) it follows that if $x\in E$ is an endpoint,  then    $\hat d_n=0$ and $\eta_n(x)=1$ for all large integers $n$, and so
  $
 \Theta^{*s}(\mu,x)={2^{-s}}.
 $

\end{itemize}
\end{remark}

 Note by Theorem \ref{main:densities} (iii) that for $\mu$-almost every $x\in E$ we have
 $\Theta_*^s(\mu, x)<\Theta^{*s}(\mu, x)$. This implies that $E$ is \emph{irregular} (see \cite[Chapter 5]{Falconer_1990} for its definition). In fact,  the pointwise lower and upper densities are distinct at every point of $E$.  Observe that $\ga_n(x)\in[0, 1]$ and $\eta_n(x)\in[1/2, 1]$ for any $x\in E$. So,  by Theorem \ref{main:densities} it follows that
 \begin{equation}\label{eq:bounds}
 (2R/\rho)^{-s}\le \Theta_*^s(\mu,x)\le \left(2R/\rho-2\right)^{-s}\quad<\quad2^{-s}\le \Theta^{*s}(\mu, x)\le 1
 \end{equation}
 for all $x\in E$, where the third inequality follows by Lemma \ref{lem:inequality} (see below) and the last inequality is well-known (see \cite[Theroem 5.1]{Falconer_1990} or \cite{Mattila_1995}). In fact, when $N$ is even, the upper bound for $\Theta^{*s}(\mu, x)$ can be refined  to $(NR)^{-s}<1$ (see Lemma \ref{lem:upper-bound-density} below). This motivates us to study the sets
 \[
 E_*(a):=\set{x\in E: \Theta_*^s(\mu, x)\ge a}\quad\textrm{and}\quad E^*(b):=\set{x\in E: \Theta^{*s}(\mu, x)\le b}
 \]
with $a, b\in\R$.

 Clearly, the set-valued map $a\mapsto E_*(a)$ is non-increasing. By (\ref{eq:bounds}) it follows that $E_*(a)=E$ if $a\le (2R/\rho)^{-s}$, and  $E_*(a)=\emptyset$ if $a>(2R/\rho-2)^{-s}$. So, there must exist a critical value $a_c\in((2R/\rho)^{-s}, (2R/\rho-2)^{-s})$ such that $\dim_H E_*(a)>0$ if and only if $a<a_c$. Similarly, the set-valued map $b\mapsto E^*(b)$ is non-decreasing. Again by (\ref{eq:bounds}) it gives that $E^*(b)=\emptyset$ if $b<2^{-s}$, and $E^*(b)=E$ if $b\ge 1$. So there  exists a critical value $b_c\in(2^{-s}, 1)$ such that
 $\dim_H E^*(b)>0$ if and only if $b<b_c$.

 Our next   result is to describe the  {critical values}  $a_c$ and $b_c$ for the sets  $E_*(a)$ and $E^*(b)$, respectively. Inspired by some works on the critical values in unique beta expansions (cf.~\cite{Glendinning_Sidorov_2001, Kong_Li_Dekking_2010}) and open dynamics (cf.~\cite{Kalle-Kong-Langeveld-Li-18}),
   we   introduce two Thue-Morse type sequences $(\theta_i)$ and $(\la_i)$ in $\set{0,1,\ldots, N-1}^\N$. For a word $\mathbf c=c_1\ldots c_n$ we define its \emph{reflection} by $\overline{\mathbf c}:=(N-1-c_1)\ldots (N-1-c_n)$. Furthermore, if $c_n<N-1$, then we write $\mathbf c^+:=c_1\ldots c_{n-1}(c_n+1)$; and if $c_n>0$, then we set $\mathbf c^-:=c_1\ldots c_{n-1}(c_n-1)$.

\begin{definition}
\label{def:sequences}
The sequences  $(\la_i)_{i=1}^\f$ and $(\theta_i)_{i=1}^\f$  are defined recursively as follows. Set $\la_1=\theta_1=N-1$, and if $\la_1\ldots \la_{2^n}$ and $\theta_1\ldots \theta_{2^n}$ are defined for some $n\ge 0$, then
\[\la_{2^n+1}\ldots\la_{2^{n+1}}=\overline{\la_1\ldots\la_{2^n}}^+\quad\textrm{and}\quad  \theta_{2^n+1}\ldots \theta_{2^{n+1}}=\overline{\theta_1\ldots\theta_{2^n}}.\]
\end{definition}
By Definition \ref{def:sequences} it follows that the sequence $(\la_i)\in\set{0,1, N-2, N-1}^\N$   begins with
\[
(\la_i)=(N-1)10(N-1)\,0(N-2)(N-1)1\; 0(N-2)(N-1)0(N-1)10(N-1)\ldots,
\]
and the sequence
$(\theta_i)\in\set{0, N-1}^\N$   begins with
\[
(\theta_i)=(N-1)00(N-1)\, 0(N-1)(N-1)0\; 0(N-1)(N-1)0 (N-1)00(N-1)\ldots.
\]
 We emphasize that for $N=2$ the sequence $(\la_i)$ is the shift of the classical Thue-Morse sequence, and the sequence $(1-\theta_i)$ is indeed the Thue-Morse sequence (cf.~\cite{Alloche_Shallit_2003}).

Now we state our second  result on  the critical values  of  $E_*(a)$ and $E^*(b)$, respectively.

\begin{main}\label{main:critical-values}
Let $N\ge 2$ and $0<\rho\le 1/N^2$.
\begin{enumerate}
\item[(A)] The critical value for   $E_*(a)$ is given by
\[
a_c=\left(2 \Big(\frac{R}{\rho}-1+R\sum_{i=1}^\f \la_{i+1}\rho^{i-1}\Big)\right)^{-s}.
\]
That is
(i) if $a<a_c$, then $\dim_H E_*(a)>0$;
(ii) if $a=a_c$, then  $E_*(a)$ is uncountable;
and (iii) if $a>a_c$, then $E_*(a)$ is at most countable.

\item[(B)] The critical value for $E^*(b)$ is given by
\[
b_c=\left(2R\sum_{i=1}^\f\theta_i\rho^{i-1}\right)^{-s}.
\] That is
(i)  if $b<b_c$, then $E^*(b)$ is at most countable;
(ii) if $b=b_c$, then   $E^*(b)$ is uncountable;
and (iii) if $b>b_c$, then $\dim_H E^*(b)>0$.

\end{enumerate}
\end{main}
Note by (\ref{eq:bounds}) that $(2R/\rho)^{-s}\le a_c\le (2R/\rho-2)^{-s}<2^{-s}\le b_c\le 1$. By some numerical calculation we list the values of $(2R/\rho)^{-s}, a_c, (2R/\rho-2)^{-s}, 2^{-s}$ and $b_c$ for $2\le N\le 8$ and $\rho=1/N^2$ (see Table \ref{tab:1}).


\begin{table}[h!]
  \centering
  \begin{tabular}{c|c|c|c|c|c|c|c}
    \hline
  N& 2 & 3 & 4 & 5 & 6 & 7 & 8 \\\hline
$(2R/\rho)^{-s}\approx$& 0.408248 & 0.353553 & 0.316228 & 0.288675 & 0.267261 & 0.25 & 0.235702 \\\hline
 $a_c \approx$& 0.422744 & 0.38039 & 0.340358 & 0.308856 & 0.284091 & 0.26419 & 0.247828 \\\hline
 $(2R/\rho-2)^{-s}\approx$&0.5 & 0.408248 & 0.353553 & 0.316228 & 0.288675 & 0.267261 & 0.25 \\\hline
 $2^{-s}\approx$&0.707107 & 0.707107 & 0.707107 & 0.707107 & 0.707107 & 0.707107 & 0.707107 \\\hline
 $b_c\approx$& 0.809703 & 0.749479 & 0.730207 & 0.721665 & 0.717129 & 0.714431 & 0.712695 \\
    \hline
  \end{tabular}
  \caption{The list of the  values $(2R/\rho)^{-s}, a_c, (2R-2)^{-s}, 2^{-s}$ and $b_c$ with $N=2,3,\ldots, 8$ and $\rho=1/N^2$. In this case we have $s=-\log N/\log\rho=1/2$ and $R=(1-\rho)/(N-1)=N^{-1}+N^{-2}$. }
  \label{tab:1}
\end{table}

	The rest of the paper is arranged as follows. In Section \ref{s2} we describe  the pointwise lower and upper densities of $\mu$ at each $x\in E$ and prove Theorem \ref{main:densities}. In Section \ref{sec:critical-values} we determine  the critical values of $E_*(a)$ and $E^*(b)$ respectively, and prove Theorem \ref{main:critical-values}. 

	\section{Pointwise densities of $\mu$}\label{s2}

	In this section we will describe the pointwise densities of $\mu$ at any point $x\in E$, and prove Theorem \ref{main:densities}.
Note that  $E\subseteq \cup_{i=0}^{N-1}f_i([0, 1])$ and the union is pairwise disjoint. Then by (\ref{eq:mu}) it follows that the measure $\mu$ has the same weight $1/N$ on each basic interval $f_i([0, 1])$. This is a special case of  the following  lemma.
\begin{lemma}\label{lem:density-1}
Let $N\ge 2$ and $0<\rho\le 1/N^2$. Then
for any measurable subset $A\subset (-1, 2)$ and any $d_1\ldots d_n\in\set{0,1,\ldots, N-1}^n$ we have
\[
\mu\big(f_{d_1\ldots d_n}(A)\big)=\frac{\mu(A)}{N^n}.
\]
\end{lemma}
\begin{proof}
Let $A\subset(-1, 2)$ be a $\mu$-measurable set. Recall that $R=(1-\rho)/(N-1)$. Then  for each $i\in\set{0,1,\ldots, N-1}$ we have
$
f_i(A)\subset\left(-\rho+i R, 2\rho+i R\right).
$
Since $\rho\le 1/N^2$ and $f_j(E)\subset f_j([0, 1])$, one can easily verify that
\[
f_i(A)\cap f_j(E)=\emptyset\quad\forall~i\ne j.
\]
Then by (\ref{eq:mu}) this implies that
$
\mu(f_i(A))={\mu(A)}/{N}.
$
So, by induction on $n$ it follows that
\[
\mu(f_{d_1\ldots d_n}(A))=\frac{\mu(A)}{N^n}\quad\textrm{for all }n\in\N,
\]
completing the proof.
\end{proof}
In the following we give the   bounds for $\mu([0, t])$ and $\mu([t, 1])$, which plays an important role in describing the densities of $\mu$.

\begin{lemma}\label{lem:density-2}
Let $N\ge 2$ and $0<\rho\le 1/N^2$. Then for any $t\in [0, 1]$ we have
\[
\left(\frac{\rho}{R}\right)^s\cdot t^s\le \mu([0, t])\le t^s\quad\textrm{and}\quad \left(\frac{\rho}{R}\right)^s(1-t)^s\le \mu([t, 1])\le (1-t)^s.
\]
\end{lemma}	
\begin{proof}
Let $t\in [0, 1]$. We only prove the bounds for $\mu([0, t])$, since the bounds for $\mu([t,1])$ can be obtained by using the symmetry of $\mu$ that $\mu(B(x,r))=\mu(B(1-x, r))$ for any $x\in E$ and any $r>0$.

Observe that $\mu=\mathcal H^s|_E$. Then  the upper bound for $\mu([0, t])$  follows directly  from the work of Zhou \cite{Zhou-1998} that $\mathcal H^s(E\cap U)\le |U|^s$ for any $U\subset \R$. So it suffices to prove the lower bound.
 Clearly it  holds for $t=0$. In the following we assume $t\in (0, 1]$. First we consider $t\in E$.  Then by (\ref{eq:proj-map}) there exists a sequence $(d_i)\in\set{0,1,\ldots, N-1}^\N$ such that
$
t=R\sum_{i=1}^\f d_i\rho^{i-1}.
$
We claim that
\begin{equation*}
\mu([0,t])=\sum_{i=1}^\f \frac{d_i}{N^i}.
\end{equation*}

To prove the claim it suffices to prove for any $k\in\N$ that
\begin{equation}\label{eq:kong-01}
\mu([0, t_k])=\sum_{i=1}^k\frac{d_i}{N^i}\quad\textrm{when}\quad t_k=R\sum_{i=1}^k d_i\rho^{i-1}.
\end{equation}
We will prove (\ref{eq:kong-01}) by induction on $k$.
Clearly, for $k=1$ we have $t_1=d_1R$, and then $[0, t_1]$ contains exactly $d_1$ basic intervals of level $1$, i.e.,
\[
\bigcup_{i=0}^{d_1-1}f_i(E)\subset [0, t_1].
\]
So by Lemma \ref{lem:density-1} it follows that $\mu([0, t_1])\ge d_1/N$. On the other hand, $[0, t_1]\cap f_{d_1}(E)$ is a singleton, and $[0, t_1]\cap f_j(E)=\emptyset$ for any $j>d_1$. Therefore, we conclude that $\mu([0, t_1])=d_1/N$. This proves (\ref{eq:kong-01}) for $k=1$.

Now   suppose (\ref{eq:kong-01}) holds for $k=n$, and we consider $k=n+1$. Note that
\[
t_{n+1}=R\sum_{i=1}^{n+1}d_i\rho^{i-1}=t_n+R d_{n+1}\rho^n.
\]
Since each basic interval of level $n$ has length $\rho^n$, this implies that
\[
[t_n, t_{n+1})\cap E=\bigcup_{j=0}^{d_{n+1}-1} f_{d_1\ldots d_n j}(E).
\]
So, by the induction hypothesis and Lemma \ref{lem:density-1} we obtain
\begin{align*}
\mu([0, t_{n+1}])&=\mu([0, t_n])+\mu([t_n, t_{n+1}])=\sum_{i=1}^{n}\frac{d_i}{N^i}+\frac{d_{n+1}}{N^{n+1}}=\sum_{i=1}^{n+1}\frac{d_i}{N^i}.
\end{align*}
This proves (\ref{eq:kong-01}) for $k=n+1$, and then the claim follows by induction.

By the claim and using  $N^{-1}=\rho^s$ it follows that
\[
\frac{\mu([0, t])}{t^s}=\frac{\sum_{i=1}^\f d_i\rho^{i s}}{\left(\frac{R}{ \rho}\right)^s\left(\sum_{i=1}^\f d_i\rho^i\right)^s}\ge \left(\frac{\rho}{R}\right)^s,
\]
where the inequality follows by using $s\in(0, 1)$ and the basic inequality
\[
\left(\sum_{i=1}^\f x_i\right)^s\le \sum_{i=1}^\f x_i^s\quad\forall~ x_i\ge 0.
\]

Now for $t\in(0,1]\setminus E$ let $t'$ be the smallest element in $E$ strictly larger than $t$. Then
$\mu([0, t])=\mu([0, t'])$, and therefore,
\[
\mu([0, t])=\mu([0, t'])\ge \left(\frac{\rho}{R}\right)^s\cdot (t')^s>\left(\frac{\rho}{R}\right)^s\cdot t^s.
\]
This completes the proof.
\end{proof}	

The following lemma is elementary but turns out to be useful in our proofs later.
\begin{lemma}
  \label{lem:inequality}
  Let $N\ge 2$ and $0<\rho\le 1/N^2$.   Then
    \[
  0<\frac{\rho}{R}\le \frac{1}{N+1}.
  \]
  \begin{enumerate}[{\rm(i)}]
  \item Let  $A, B\ge 0$.  If
  $
  \rho^{1-s}A<B,
 $ then the function
  \[
  g(r)=\frac{A+(r-B)^s}{r^s}
  \]
  is  strictly increasing in $[B, B+\rho]$.

 \item Let $m\ge 1$. Then for any $t\in[0, \rho]$ the sequences
 \[
 a_j:=\frac{j}{(j R-t)^s}\quad\textrm{and} \quad b_j:=\frac{1+m j}{(jR+t)^s}, \quad j\ge 1,
 \]
are both strictly   increasing.

  \end{enumerate}
\end{lemma}
\begin{proof}
  Note that $(N-1)R=1-\rho\ge N^2\rho-\rho>0$. This gives
  \[
  0<\frac{\rho}{R}\le \frac{1}{N+1}.
  \]
  For (i) we observe that
   \begin{equation}\label{eq:mar-28-1}
  g'(r)>0\quad\Longleftrightarrow\quad (r-B)^{1-s}A<B.
  \end{equation}
  Note that  $r\in[B, B+\rho]$ and $s\in(0, 1)$. So, if $\rho^{1-s}A<B$, by (\ref{eq:mar-28-1}) it follows that $g'(r)>0$ for all $r\in[B, B+\rho]$. This proves (i).

  For (ii) we note  for $j\ge 1$ that $a_{j+1}>a_j$ if and only if
  \begin{equation}\label{eq:mar-28-2}
  \frac{j+1}{j}>\left(\frac{(j+1)R-t}{jR-t}\right)^s.
  \end{equation}
  Since $\rho\in(0, 1/N^2]$, we have $s=-\log N/\log \rho\in(0, 1/2]$. Thus, (\ref{eq:mar-28-2}) holds if it holds for $s=1/2$, i.e.,
  \[
  \frac{j+1}{j}>\left(\frac{jR-t+R}{jR-t}\right)^{1/2}.
  \]
  By rearrangements and using $t\in[0, \rho]$ we only need to prove
\[
(2j+1)\left(j-\frac{\rho}{R}\right)>j^2.
\]
Since $0<\rho/R\le 1/(N+1)$,  it suffices to verify $(2j+1)(j-\frac{1}{N+1})>j^2$, which follows directly by using  $j\ge 1$.

Similarly, one can verify that $b_{j+1}>b_j$ for all $j\ge 1$.
  \end{proof}

%
	
\subsection{Lower density $\Theta_*(\mu, x)$}	
	In the following we give the lower bound for $\mu(B(x, r))/(2r)^s$, where $B(x, r)=(x-r, x+r)$ is the open ball of radius $r$ at center $x$. Based on Lemma \ref{lem:density-2} we first consider $x\in f_i([0, 1])$ for $i=0$ and $i=N-1$.
	
	\begin{lemma}\label{lem:lower-bound}
	Let $N\ge 2$ and $0<\rho\le 1/N^2$.
\begin{enumerate}[{\rm(i)}]
\item
	If $x\in[0, \rho]$, then for any $\max\set{x, \rho-x}\le r\le 1-x$ we have
	\[
	\frac{\mu(B(x, r))}{(2r)^s}\ge \frac{\rho^s}{2^s(R-x)^s},
	\]
	where the equality holds for $r=R-x$.

\item If $x\in[1-\rho, 1]$, then for any $\max\set{1-x, \rho-(1-x)}\le r\le x$ we have
\[
\frac{\mu(B(x, r))}{(2r)^s}\ge \frac{\rho^s}{2^s(R-(1-x))^s},
\]
where the equality holds for $r=R-(1-x)$.
\end{enumerate}
	\end{lemma}
	\begin{proof}
Note  that $\mu(B(x,r))=\mu(B(1-x, r))$ for any $x\in E$ and $r>0$. So, (ii) can be deduced from (i) by   the symmetry of $\mu$. In the following we only   prove (i).

	Let $x\in[0, \rho]$, and take $\max\set{x, \rho-x}\le r\le 1-x$. Then $B(x, r)\subset(-1, 2)$, and $B(x, r)$ nearly contains the interval $[0, \rho]$  which means $B(x, r)$ contains $[0, \rho]$ up to a point. Recall that $R=(1-\rho)/(N-1)$. We split the proof into the following two cases.
	
	Case I. $\max\set{x, (j-1)R+\rho-x}\le  r\le j R-x$ for some $j\in\set{1, 2,\ldots, N-1}$. Then $(j-1)R+\rho\le x+r\le j R$. Note that $f_{i+1}(0)-f_i(0)=R$ for any $0\le i<N-1$. Then $B(x, r)$ nearly contains $f_i([0, 1])$ for $0\le i\le j-1$, and $B(x, r)$ has no intersect with any other basic interval of level $1$. So, by Lemma \ref{lem:density-1} it follows that
	\[
	\mu(B(x, r))=j \rho^s.
	\]
	This implies
	\begin{align*}
	\frac{\mu(B(x, r))}{(2r)^s}=\frac{j\rho^s}{(2r)^s}\ge \frac{j\rho^s}{2^s\left(j R-x\right)^s},
	\end{align*}
	where the equality holds for $r=j R-x$.
	
	Case II. $j R-x< r\le j R+\rho-x$ for some $j\in\set{1,2,\ldots, N-1}$. Then $B(x, r)$ contains the level-$1$ basic intervals $f_0([0, 1]),\ldots, f_{j-1}([0, 1])$ and intersect  $f_j([0,1])$. But $B(x,r)$ has no intersect with any other level-$1$ basic interval. So, by Lemmas \ref{lem:density-1}, \ref{lem:density-2} and using the uniformity  of $\mu$ it follows that
	\begin{align*}
	\mu(B(x, r))=j \rho^s+\mu([j R, x+r])&=j \rho^s+\mu([0, x+r-j R])\ge j\rho^s+\left(\frac{\rho}{ R}\right)^s(x+r-j R)^s.
	\end{align*}
	This implies
	\begin{equation}\label{eq:feb-21-1}
	\frac{\mu(B(x, r))}{(2r)^s}\ge \frac{\rho^s}{(2R)^s}g_1(r),\quad\textrm{where}\quad g_1(r):=\frac{j R^s+(x+r-j R)^s}{r^s}.
	\end{equation}
We claim that $g_1(r)$ is strictly increasing in $(jR-x, jR+\rho-x]$.

By Lemma \ref{lem:inequality} (i) with $A=j R^s$ and $B=jR-x$ it suffices to check $\rho^{1-s}j R^s<jR-x$. Since $s\in(0, 1/2]$ and $x\in[0, \rho]$, it suffices to verify
\[
j\rho\left(\frac{R}{\rho}\right)^{1/2}<jR-\rho.
\]
Dividing $R$ on both sides  reduces to
\[
 j\left(\frac{\rho}{R}\right)^{1/2}<j-\frac{\rho}{R}.
\]
This can be easily verified by  using $j\ge 1$ and  $0<\rho/R\le 1/(N+1)$ (see Lemma \ref{lem:inequality}). So, we establish the claim.

Therefore, by  (\ref{eq:feb-21-1}) and the claim it follows that
\[
\frac{\mu(B(x, r))}{(2r)^s}\ge \frac{\rho^s}{(2R)^s}g_1(r)> \frac{\rho^s}{(2R)^s}g_1(j R-x)= \frac{j\rho^s}{2^s\left(j R-x\right)^s}
\]
for any $r\in(jR-x, jR+\rho-x]$.
 Finally, observe by Lemma \ref{lem:inequality} (ii)   that the sequence
\[
\hat a_j=\frac{j}{(j R-x)^s}, \quad j=1,2,\ldots, N-1,
\]
is strictly increasing. This proves (i).
 	\end{proof}

In the following we will determine the lower bound of $\mu(B(x, r))/(2r)^s$ for $x\in f_k([0, 1])$ with $k=1,2,\ldots, N-2$.
\begin{proposition}\label{prop:lower-bound}
Let $N\ge 2$ and $0<\rho\le 1/N^{2}$.
\begin{enumerate}
\item[{\rm(i)}] If $x\in[k R, kR+\rho/2]$ for  some $k\in\set{1,\ldots, N-2}$, then for any $k R+\rho-x \le r\le \max\set{x, 1-x}$ we have
\[
\frac{\mu(B(x, r))}{(2r)^s}\ge \frac{\rho^s}{2^s(R-(kR+\rho-x))^s},
\]
where the equality holds when $r=R-(kR+\rho-x)$.

\item[{\rm(ii)}] If $x\in[k R+\rho/2, k R+\rho]$ for some $k\in\set{1,\ldots, N-2}$, then for any $x-k R \le  r\le \max\set{x, 1-x}$ we have
\[
\frac{\mu(B(x, r))}{(2r)^s}\ge \frac{\rho^s}{2^s(R-(x-k R))^s},
\]
where the equality holds when $r=R-(x-k R)$.
\end{enumerate}
\end{proposition}
\begin{proof}
Observe   for $k\in\set{1,\ldots, N-2}$ that $x\in[k R+\rho/2, kR+\rho]$ if and only if $1-x\in[(N-1-k)R, (N-1-k)R+\rho/2]$. So, (ii) can be deduced  from (i) by using the symmetry of $\mu$. In the following we only prove (i).

 Take $x\in[k R, kR+\rho/2]$ with $k\in\set{1, 2,\ldots, N-2}$. Then for $kR+\rho-x \le r\le x$ the ball $B(x, r)$ nearly contains the basic interval $f_k([0, 1])$. Using the uniformity of $\mu$ it follows that
\[
\mu(B(x, r))\ge \mu(B((N-1-k)R+x, r)).
\]
Observe that $(N-1-k)R+x\in[1-\rho, 1]$. By Lemma \ref{lem:lower-bound} (ii) it follows that
\begin{equation}\label{eq:feb-1}
\begin{split}
\frac{\mu(B(x, r))}{(2r)^s}\ge \frac{\mu(B((N-1-k)R+x, r))}{(2r)^s}&\ge\frac{\rho^s}{2^s(R-(1-(N-1-k)R-x))^s}\\
&= \frac{\rho^s}{2^s(R-(kR+\rho-x))^s},
\end{split}
\end{equation}
where the equalities  hold for $r=R-(kR+\rho-x)$.

If $x\ge 1/2$, then $\max\set{x,1-x}=x$ and we are done. Now suppose $x<1/2$. Then for $x<r\le 1-x$ we have
$\mu(B(x, r))=\mu(B(0, x+r))$.
By Lemma \ref{lem:density-2} it follows that
\begin{equation}\label{eq:mar-27-2}
\frac{\mu(B(x, r))}{(2r)^s}\ge\frac{\rho^s}{2^s}\cdot\frac{(x+r)^s}{(R r)^s}.
\end{equation}
Using $x\ge kR,  r\le 1-x$ and $0<\rho\le 1/N^2$, one can verify that
\[
\frac{x+r}{R r}>\frac{1}{R-(k R+\rho-x)}\quad \textrm{for all }1\le k\le N-2.
\]
So, by (\ref{eq:feb-1}) and (\ref{eq:mar-27-2}) we establish  (i).
\end{proof}

By Lemma \ref{lem:lower-bound}  and  Proposition \ref{prop:lower-bound} it follows that  for any $x\in E$, if   $B(x, r)$ contains at least one level-$1$ basic interval  but  it does not contain the unit interval $[0, 1]$,  then
\begin{equation}\label{eq:feb-19-1}
\frac{\mu(B(x, r))}{(2r)^s}\ge \frac{\rho^s}{2^s(R-S(x))^s},
\end{equation}
where
\begin{equation*}\label{eq:map-s}
S(x)=\left\{\begin{array}{lll}
x&\textrm{if}& 0\le x\le \rho,\\
\max\set{x-k R, kR+\rho-x}&\textrm{if}& k R\le x\le kR+\rho\textrm{ for }1\le k\le N-2,\\
1-x&\textrm{if}& 1-\rho\le x\le 1.
\end{array}\right.
\end{equation*}
Recall that  the $N$-to-$1$ expanding map $T: \bigcup_{k=0}^{N-1}[kR, kR+\rho]\to [0, 1]$   satisfies
	\[
	T(x)= \frac{x-k R}{\rho}\quad\textrm{if}\quad k R\le x\le kR+\rho,
	\]
where $k=0,1,\ldots, N-1$.
So for any
$
x=\pi(d_1d_2\ldots)\in E
$
we have
$
T^n(x)=\pi(d_{n+1}d_{n+2}\ldots).
$

\begin{proof}[Proof of Theorem \ref{main:densities} (i)]
Take $x=\pi(d_1d_2\ldots)\in E$ and $r\in(0, \rho)$. Then there exists $n\in\N$ such that $B(x, r)$ contains the level-$(n+1)$ basic interval $f_{d_1\ldots d_{n+1}}([0, 1])$, but it  does not contain the level-$n$ basic interval $f_{d_1\ldots d_{n}}([0, 1])$. This implies that
\[
T^{n}(B(x, r) )=(f_{d_1\ldots d_{n}})^{-1}(B(x, r) )\supseteq f_{d_{n+1}}([0,1]),
\]
but $T^{n}(B(x, r))$ does not contain $[0,1]$.

Let $y=T^{n}x=\pi(d_{n+1}d_{n+2}\ldots)$, and let $r'=\rho^{-n}r$. Then $y\in f_{d_{n+1}}([0,1])$,  $\rho/2<r'<1$ and $T^{n}(B(x, r) )=B(y, r') $. By Lemma \ref{lem:density-1} and (\ref{eq:feb-19-1}) it follows that
\begin{equation}\label{eq:feb-19-2}
\begin{split}
\frac{\mu(B(x, r))}{(2r)^s}=\frac{(\rho^s)^{n}\mu(B(y, r'))}{(2r)^s}&=\frac{\mu(B(y, r'))}{(2r')^s}\\
&\ge \frac{\rho^s}{2^s(R-S(y))^s}\\
&=\frac{\rho^s}{2^s(R-S(T^{n}x))^s}=\frac{\rho^s}{2^s(R-\rho\ga_{n+1}(x))},
\end{split}
\end{equation}
where the last equality holds by the following observation: if $1\le d_{n+1}\le N-2$, then
\begin{align*}
S(T^n x)&=\max\set{T^n x-d_{n+1}R, d_{n+1}R+\rho-T^n x}\\
&=\rho\max\set{T^{n+1}x, 1-T^{n+1}x}=\rho\ga_{n+1}(x);
\end{align*}
if $d_{n+1}=0$, then $S(T^n x)=T^n x=\rho T^{n+1}x=\rho\ga_{n+1}(x)$; and if $d_{n+1}=N-1$, then $S(T^n x)=1-T^n x=\rho(1-T^{n+1}x)=\rho\ga_{n+1}(x)$.
Letting $r\to 0$ in (\ref{eq:feb-19-2}), and then $n\to\f$, we deduce that
\begin{equation}\label{eq:feb-19-3}
\Theta_*^s(\mu, x)=\liminf_{r\to 0}\frac{\mu(B(x, r))}{(2r)^s}\ge \frac{\rho^s}{2^s(R-\rho\liminf_{n\to\f} \ga_n(x))^s}.
\end{equation}

On the other hand, observe that the  equality holds  in (\ref{eq:feb-19-2})  when $r'=R-S(T^{n}x)$, or equivalently, when
$r=\rho^{n}(R-S(T^{n}x)).$ So, let $r_n:=\rho^{n}(R-S(T^{n}x))$. Then $r_n\to 0$ as $n\to\f$. Letting $r\to 0$ along the subsequence $(r_n)$ in (\ref{eq:feb-19-3}) we conclude that
\[
\Theta_*^s(\mu, x)=\frac{\rho^s}{2^s(R-\rho\liminf_{n\to\f}\ga_n(x))^s}=\frac{1}{2^s(R/\rho-\liminf_{n\to\f}\ga_n(x))^s}.
\]
This completes the proof.
\end{proof}

\subsection{Upper density $\Theta^*(\mu, x)$}
Using the same idea as in the previous subsection we are able to determine the pointwise upper density of $\mu$ at each $x\in E$. But the calculation is more involved.
First we consider $x\in f_i([0,1])$ for $i=0$ and $i=N-1$.
\begin{lemma}\label{lem:upper-density-1}
Let $N\ge 2$ and $0<\rho\le 1/N^2$.
\begin{enumerate}[{\rm(i)}]
\item
 Let  $x\in[0, \rho]$. Then for any $\max\set{x, \rho-x}\le r\le 1-x$ we have

\[
\frac{\mu(B(x,r))}{(2r)^s}\le\left\{ \begin{array}{lll}
\frac{\rho^s}{2^s(\rho-x)^s}&\textrm{if}&x\in[0, \rho/2],\\
\max\set{\frac{\rho^s}{2^s x^s}, \frac{1}{2^s(1-x)^s}}&\textrm{if}& x\in[\rho/2, \rho],
\end{array}\right.
\]
where the equality is   attainable.

 \item  Let $x\in[1-\rho, 1]$. Then for any $\max\set{1-x, \rho-(1-x)}\le r\le x$ we have
\[
\frac{\mu(B(x,r))}{(2r)^s}\le\left\{\begin{array}{lll}
 \max\set{\frac{\rho^s}{2^s(1-x)^s}, \frac{1}{2^s x^s}}&\textrm{if}& x\in[1-\rho, 1-\rho/2],\\
 \frac{\rho^s}{2^s (\rho-(1-x))^s}&\textrm{if}& x\in[1-\rho/2, 1],
 \end{array}\right.
\]
where the equality is attainable.

\end{enumerate}

\end{lemma}
\begin{proof}
Since the measure $\mu$ is symmetric on $E$, (ii) can be deduced from (i). In the following we only prove (i).
Let $x\in[0, \rho]$ and $\max\set{x,\rho-x}\le r\le 1-x$. Then $B(x, r)$ nearly contains $[0, \rho]$, and $B(x, r)\subset(-1, 2)$. Note that $R=(1-\rho)/(N-1)$. We consider the following two cases.

Case I. $\max\set{x, (j-1)R+\rho-x}\le  r\le jR-x$ for some $j\in\set{1,2,\ldots, N-1}$. By the same argument as in the proof of Lemma \ref{lem:lower-bound} we have $\mu(B(x,r))=j\rho^s$. So, if $x\le (j-1)R+\rho-x$, then $\max\set{x, (j-1)R+\rho-x}=(j-1)R+\rho-x$, and so
\begin{equation}\label{eq:feb-28-0}
\frac{\mu(B(x, r))}{(2r)^s}\le\frac{j\rho^s}{2^s((j-1)R+\rho-x)^s},
\end{equation}
where the equality holds for $r=(j-1)R+\rho-x$. If $x>(j-1)R+\rho-x$, then $j=1$ and $x>\rho/2$. In this case we have
\begin{equation}\label{eq:feb-28-1}
\frac{\mu(B(x,r))}{(2r)^s}\le \frac{\rho^s}{2^s x^s}
\end{equation}
where the equality holds for $r=x$.

Case II. $jR-x<r\le jR+\rho-x$ for some $j\in\set{1,2,\ldots, N-1}$. Then by Lemma \ref{lem:density-2} one can verify that
\[
\mu(B(x, r))=j\rho^s+\mu([jR, x+r])\le j\rho^s+(x+r-jR)^s.
\]
This implies
{\[
\frac{\mu(B(x,r))}{(2r)^s}\le\frac{1}{2^s}\cdot\frac{j\rho^s+(x+r-jR)^s}{r^s}=:\frac{1}{2^s} g_2(r).
\]
By Lemma \ref{lem:inequality} (i) with $A=j\rho^s$ and $B=jR-x$ one can easily verify   that $\rho^{1-s}A<B$, and thus $g_2$ is   strictly increasing in $(jR-x, jR+\rho-x]$.} So, for any $jR-x<r\le jR+\rho-x$,
\begin{equation}\label{eq:feb-28-2}
\frac{\mu(B(x,r))}{(2r)^s}\le\frac{1}{2^s}g_2(jR+\rho-x)=\frac{(j+1)\rho^s}{2^s(jR+\rho-x)^s},
\end{equation}
where the equality holds for $r=jR+\rho-x$.

{Note by Lemma \ref{lem:inequality} (ii)   that the sequence
\[
\hat b_j:=\frac{j+1}{(jR+\rho-x)^s},\quad j=1,\ldots, N-1
\]
is  strictly increasing in $j$.} Therefore, by  (\ref{eq:feb-28-0}) and (\ref{eq:feb-28-2}) it follows that for $x\in[0, \rho/2]$,
\begin{align*}
\frac{\mu(B(x,r))}{(2r)^s}&\le\max\set{\frac{\rho^s}{2^s(\rho-x)^s}, \frac{N\rho^s}{2^s((N-1)R+\rho-x)^s}}\\
&=\max\set{\frac{\rho^s}{2^s(\rho-x)^s}, \frac{1}{2^s(1-x)^s}} =\frac{\rho^s}{2^s(\rho-x)^s},
\end{align*}
where  we have used $\rho^s=1/N$ and $R=(1-\rho)/(N-1)$. Moreover, by (\ref{eq:feb-28-1}) and (\ref{eq:feb-28-2}) it follows that for $x\in[\rho/2, \rho]$ we have
\[
\frac{\mu(B(x,r))}{(2r)^s}\le \max\set{\frac{\rho^s}{2^sx^s}, \frac{1}{2^s(1-x)^s}}.
\]
This completes the proof.
\end{proof}

In the next lemma we consider the upper bound of  $\mu(B(x,r))/(2r)^s$ for $x\in f_k([0,1])$ with $k=1,\ldots, N-2$.
\begin{proposition}\label{prop:upper-density}
Let $N\ge 2$ and $0<\rho\le 1/N^2$. For $k\in\set{0,\ldots, N-1}$ set $\hat k:=\min\set{k, N-1-k}$.
\begin{enumerate}
\item[{\rm(i)}] If $x\in[kR, kR+\rho/2]$ for some $k=1,\ldots, N-2$, then for any $kR+\rho-x\le r\le \max\set{x, 1-x}$ we have
\[
\frac{\mu(B(x, r))}{(2r)^s}\le \max\set{\frac{\rho^s}{2^s(kR+\rho-x)^s}, \frac{(1+2\hat k)\rho^s}{2^s(\hat k R+ (kR+\rho-x) )^s},\frac{1}{2^s(\max\set{x, 1-x})^s}},
\]
where  the equality is attainable.

\item[{\rm(ii)}] If $x\in[kR+\rho/2, kR+\rho]$ for some $k=1,\ldots, N-2$, then for any $x-kR\le  r\le \max\set{x, 1-x}$ we have
\[
\frac{\mu(B(x,r))}{(2r)^s}\le \max\set{\frac{\rho^s}{2^s(x-kR)^s}, \frac{(1+2\hat k)\rho^s}{2^s(\hat k R+ (x-kR) )^s},\frac{1}{2^s(\max\set{x, 1-x})^s}},
\]
where the equality is attainable.
\end{enumerate}

\end{proposition}

\begin{proof}
Note by the symmetry of $\mu$ that (ii) can be deduced from (i). In the following we only prove (i).
Let $x\in[kR, kR+\rho/2]$ and $kR+\rho-x\le r\le \max\set{x, 1-x}$. Then the ball $B(x, r)$ contains the basic interval $f_k([0, 1])$ and is contained in $(-1,2)$. We first assume $x\le 1/2$. Then $1\le k\le (N-1)/2$, and thus $\hat k=k$. We will prove in the following two steps that
\begin{equation}\label{eq:mar-2-1}
\frac{\mu(B(x,r))}{(2r)^s}\le\max\set{\frac{\rho^s}{2^s(kR+\rho-x)^s}, \frac{(1+2k)\rho^s}{2^s(2kR+\rho-x)^s}, \frac{1}{2^s(1-x)^s}}
\end{equation}
for any $x\in[0, 1/2]\cap[kR, kR+\rho/2]$ and $kR+\rho-x\le r\le 1-x$.

\medskip

{\bf Step I}. In this step we will prove that for each $0\le j\le k-1$ and for any $(k+j)R+\rho-x\le r\le (k+j+1)R+\rho-x$ we have
\begin{equation}\label{eq:step1}
\frac{\mu(B(x,r))}{(2r)^s}\le \max\set{\frac{(1+2j)\rho^s}{2^s(jR+(kR+\rho-x))^s}, \frac{(1+2(j+1))\rho^s}{2^s((j+1)R+(kR+\rho-x))^s}},
\end{equation}
where the equality is attainable.
Take $0\le j\le k-1$. Let $(k+j)R+\rho-x\le r\le (k+j+1)R+\rho-x$.
Then   $B(x, r)$ contains   the basic intervals $f_i([0,1])$ with $k-j\le i\le k+j$ and it might have intersection with $f_{k-j-1}([0, 1])$ and $f_{k+j+1}([0,1])$ (see Figure \ref{fig:2}). Since $x\in[kR, kR+\rho]$, we can partition the interval $[(k+j)R+\rho-x, (k+j+1)R+\rho-x]$ as follows:
\begin{equation}\label{eq:mar-29-0}
\begin{split}
(k+j)R+\rho-x&\le x-(k-j-1)R-\rho\\
&\le (k+j+1)R-x\\
&\le x-(k-j-1)R\le (k+j+1)R+\rho-x.
\end{split}
\end{equation}
Now we prove (\ref{eq:step1}) by considering the following four cases according to the partition.

	 \begin{figure}[h!]
\begin{center}
\begin{tikzpicture}[
    scale=7,
    axis/.style={very thick},
    important line/.style={thick},
    dashed line/.style={dashed, thin},
    pile/.style={thick, ->, >=stealth', shorten <=2pt, shorten
    >=2pt},
    every node/.style={color=black}
    ]

  \draw[axis] (-0.5,-0.2)  -- (-0.3,-0.2) node(xline)[right]{};
 \node[] at (-0.4, -0.15){$f_{k-j-1}$};
 \node[] at (-0.5, -0.25){$(k-j-1)R$};

      \draw[axis,blue] (-0.1,-0.2)  -- (0.1,-0.2) node(xline)[right]{};
 \node[] at (0, -0.15){$f_{k-j}$};
 \node[] at (-0.12, -0.25){$(k-j)R$};
 \node[blue] at (0.25, -0.2){$\cdots$};

     \draw[axis,blue] (0.4,-0.21)  -- (0.6,-0.21) node(xline)[right]{};
 \node[] at (0.5, -0.15){$f_k$};
  \node[] at (0.38, -0.25){$kR$};
  \draw[red, pile](0.45, -0.1)--(0.45,-0.21);
  \node[red]at(0.45,-0.08){$x$};

 \node[blue] at (0.75, -0.2){$\cdots$};

    \draw[axis,blue] (0.9,-0.2)  -- (1.1,-0.2) node(xline)[right]{};
 \node[] at (1, -0.15){$f_{k+j}$};
  \node[] at (0.88, -0.25){$(k+j)R$};

    \draw[axis] (1.3,-0.2)  -- (1.5,-0.2) node(xline)[right]{};
 \node[] at (1.4, -0.15){$f_{k+j+1}$};
  \node[] at (1.3, -0.25){$(k+j+1)R$};

\end{tikzpicture}
\end{center}
\caption{The ball $B(x, r)$ with $x\in[kR, kR+\rho/2]$ contains $f_{i}([0,1])$ with $k-j\le i\le k+j$ and might intersect $f_{k-j-1}([0,1])$ and $f_{k+j+1}([0,1])$. The left and right  endpoints of each basic interval $f_i([0, 1])$ are $iR$ and $iR+\rho$ respectively. }\label{fig:2}
\end{figure}
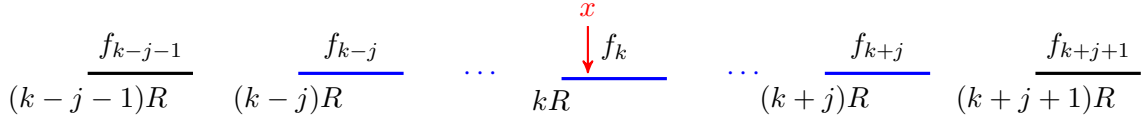

Case I. $(k+j)R+\rho-x\le r\le x-(k-j-1)R-\rho$. Then the ball $B(x, r)$ contains only the basic intervals $f_{i}([0,1])$ with $k-j\le i\le k+j$ and it has no intersect with any other basic interval of level-$1$. So,
\[
\frac{\mu(B(x,r))}{(2r)^s}=\frac{(1+2j)\rho^s}{(2r)^s}\le\frac{(1+2j)\rho^s}{2^s((k+j)R+\rho-x)^s},
\]
where the equality holds for $r=(k+j)R+\rho-x$.

Case II. $x-(k-j-1)R-\rho\le r\le (k+j+1)R-x$. Then the ball $B(x,r)$ contains the basic interval $f_i([0,1])$ with $k-j\le i\le k+j$ and it intersects $f_{k-j-1}([0,1])$. But it has no intersect with any other basic interval of level-$1$. This gives
\begin{align*}
\mu(B(x,r))&=(1+2j)\rho^s+\mu([x-r, (k-j-1)R+\rho])\\
&\le (1+2j)\rho^s+((k-j-1)R+\rho-x+r)^s,
\end{align*}
where the inequality follows by Lemma \ref{lem:density-2}. So,
\begin{equation}\label{eq:mar-29-1}
\frac{\mu(B(x,r))}{(2r)^s}\le\frac{(1+2j)\rho^s+((k-j-1)R+\rho-x+r)^s}{(2r)^s}=:\frac{1}{2^s}\hat g_1(r).
\end{equation}
 Let $A=(1+2j)\rho^s$ and $B=x-(k-j-1)R-\rho$. Note by Lemma \ref{lem:inequality} that $0<\rho/R\le 1/(N+1)$. Then by using $x\ge kR$   one can easily verify that $\rho^{1-s}A<B$. So by (\ref{eq:mar-29-0}) and Lemma \ref{lem:inequality} (i)  it follows that  the function $\hat g_1$ is strictly increasing in
$[x-(k-j-1)R-\rho, (k+j+1)R-x]$.   Therefore,
\[
\frac{\mu(B(x,r))}{(2r)^s}\le\frac{1}{2^s}\hat g_1(r)\le \frac{1}{2^s}\hat g_1\big((k+j+1)R-x\big)=\frac{(1+2j)\rho^s+(2kR+\rho-2x)^s}{2^s((k+j+1)R-x)^s}.
\]

Case III. $(k+j+1)R-x\le r\le x-(k-j-1)R$. Then the ball $B(x,r)$ not only contains the basic interval $f_{i}([0, 1])$ with $k-j\le i\le k+j$ but also intersect $f_{k-j-1}([0,1])$ and $f_{k+j+1}([0,1])$. However, it has no intersect with any other basic intervals of level-$1$. So, by Lemmas \ref{lem:density-2} and \ref{lem:lower-bound} it follows that
\begin{align*}
\frac{\mu(B(x, r))}{(2r)^s}&\le \frac{(1+2j)\rho^s+((k-j-1)R+\rho-x+r)^s+(x+r-(k+j+1)R)^s}{(2r)^s}\\
&=:\frac{1}{2^s}(\hat g_1(r)+\hat g_2(r)),
\end{align*}
where $\hat g_1$ is defined in (\ref{eq:mar-29-1}) and
\[
\hat g_2(r):=\frac{(x+r-(k+j+1)R)^s}{r^s}.
\]
Note by Case II that $\hat g_1(r)$ is increasing in $[x-(k-j-1)R-\rho, x-(k-j-1)R]$. Furthermore, by Lemma \ref{lem:inequality} (i) one can easily show that $\hat g_2(r)$ is increasing in $[(k+j+1)R-x, (k+j+1)R-x+\rho]$. Therefore, by (\ref{eq:mar-29-0}) it follows that
\[
\frac{\mu(B(x,r))}{(2r)^s}\le\frac{1}{2^s}\big(\hat g_1(x-(k-j-1)R)+\hat g_2(x-(k-j-1)R)\big)=\frac{(2+2j)\rho^s+(2x-2k R)^s}{2^s(x-(k-j-1)R)^s}
\]
for any $(k+j+1)R-x\le r\le x-(k-j-1)R$.

Case IV. $x-(k-j-1)R\le r\le (k+j+1)R+\rho-x$. Then the ball contains the basic intervals $f_{i}([0, 1])$ with $k-j-1\le i\le k+j$, and  it  intersects the basic interval $f_{k+j+1}([0, 1])$. But it has no intersection with any other basic intervals of level-$1$. So, by Lemma \ref{lem:density-2} it follows that
\[
\frac{\mu(B(x, r))}{(2r)^s}\le \frac{(2+2j)\rho^s+(x+r-(k+j+1)R)^s}{(2r)^s}=:\frac{1}{2^s}\hat g_3(r).
\]
Let $A=(2+2j)\rho^s$ and $B=(k+j+1)R-x$. Note by Lemma \ref{lem:inequality} that $\rho/R\in(0, 1/(N+1)]$. Then using $x\in[kR, kR+\rho/2]$ one can verify that $\rho^{1-s}A<B$. So, by   (\ref{eq:mar-29-0}) and Lemma \ref{lem:inequality} (i) it follows that $\hat g_3(r)$ is strictly increasing in $[x-(k-j-1)R, (k+j+1)R-x+\rho]$. Therefore,
\[
\frac{\mu(B(x, r))}{(2r)^s}\le\frac{1}{2^s}\hat g_3(r)\le \frac{1}{2^s}\hat g_3((k+j+1)R+\rho-x)=\frac{(1+2(j+1))\rho^s}{2^s((k+j+1)R+\rho-x)^s},
\]
where the equality holds for $r=(k+j+1)R+\rho-x$.

Observe that $\hat g_1$ and $\hat g_1+\hat g_2$ coincide at $r=(k+j+1)R-x$, and $\hat g_1+\hat g_2$ and $\hat g_3$ coincide at $r=x-(k-j-1)R$.
Therefore, (\ref{eq:step1}) follows by Cases I--IV.

\medskip

{\bf Step II.} Note that $kR+\rho-x\in[0, \rho]$. Then  by Lemma \ref{lem:inequality} (ii) it gives that   the sequence
\[
\hat b_j:=\frac{1+2j}{(jR+(kR+\rho-x))^s},\quad j\ge 1
\]
is strictly increasing. So, by (\ref{eq:step1}) in Step I it follows that for any $r\in[kR+\rho-x, 2k R+\rho-x]$,
\begin{equation}\label{eq:mar-3-4}
\frac{\mu(B(x, r))}{(2r)^s}\le \max\set{\frac{\rho^s}{2^s(kR+\rho-x)^s}, \frac{(1+2k)\rho^s}{2^s(k R+(kR+\rho-x))^s}},
\end{equation}
where the equality is attainable. If $k=(N-1)/2$, then $2k R+\rho-x=1-x$ and $(1+2k)\rho^s=N\rho^s=1$, and therefore (\ref{eq:mar-3-4}) gives (\ref{eq:mar-2-1}).

In the following it suffices to consider $k<(N-1)/2$. Note that for $2kR+\rho-x<r\le 1-x$ we have $x-r<0$. Then by the same argument as in the proof of Lemma \ref{lem:upper-density-1} and using   that the sequence
\[
\tilde b_j:=\frac{(1+2k+j)\rho^s}{2^s((k+j)R+(kR+\rho-x))^s},\quad j\ge 1
\]
is strictly increasing, it follows that for any $2kR+\rho-x\le r\le 1-x$ we have
\begin{equation}\label{eq:mar-3-5}
\begin{split}
\frac{\mu(B(x,r))}{(2r)^s}&\le \max\set{\frac{(1+2k)\rho^s}{2^s(kR+(kR+\rho-x))^s}, \frac{N\rho^s}{2^s((N-1)R+\rho-x)^s}}\\
&=\max\set{\frac{(1+2k)\rho^s}{2^s(kR+(kR+\rho-x))^s}, \frac{1}{2^s(1-x)^s}},
\end{split}
\end{equation}
where the equality is attainable.
Therefore, (\ref{eq:mar-2-1}) follows by (\ref{eq:mar-3-4}) and (\ref{eq:mar-3-5}).

\medskip

If $x\in(1/2, 1]$, then $k> (N-1)/2$ and thus $\hat k=N-1-k$. By the same argument as above we can prove that for $kR+\rho-x\le r\le (k+\hat k)R+\rho-x$,
\begin{equation}\label{eq:mar-3-6}
\begin{split}
\frac{\mu(B(x,r))}{(2r)^s}&\le \max\set{\frac{\rho^s}{2^s(kR+\rho-x)^s}, \frac{(1+2\hat k)\rho^s}{2^s(\hat k R+(kR+\rho-x))^s}  }.
\end{split}
\end{equation}
Since  $x>1/2$, for $1-x< r\le x$ we have $x+r>1$. By the same argument as in Lemma \ref{lem:upper-density-1} it follows that
\begin{equation}\label{eq:may-6-1}
\frac{\mu(B(x,r))}{(2r)^s}\le \max\set{ \frac{(1+2\hat k)\rho^s}{2^s(\hat k R+(kR+\rho-x))^s}, \frac{1}{2^s x^s}}
\end{equation}
for any $1-x\le r\le x$. Note that $(k+\hat k)R+\rho-x=(N-1)R+\rho-x=1-x$.
Therefore, by (\ref{eq:mar-3-6}) and (\ref{eq:may-6-1}) it follows that for any $x\in[kR, kR+\rho/2]\cap(1/2, 1]$ and  $kR+\rho-x\le r\le x$,
\begin{equation}\label{eq:mar-3-7}
\frac{\mu(B(x,r))}{(2r)^s}\le\max\set{\frac{\rho^s}{2^s(kR+\rho-x)^s}, \frac{(1+2\hat k)\rho^s}{2^s(\hat k R+(kR+\rho-x))^s},  \frac{1}{2^s x^s}},
\end{equation}
where the equality is attainable.

 Hence, by (\ref{eq:mar-2-1}) and (\ref{eq:mar-3-7}) we prove (i).
\end{proof}
By Lemma  \ref{lem:upper-density-1}  it follows that Proposition \ref{prop:upper-density} also holds for  $k=0$ and $N-1$. So, if $x\in[kR, kR+\rho]$ for some $k\in\set{0, 1,\ldots, N-1}$ and $\max\set{kR+\rho-x, x-kR}\le r\le \max\set{x, 1-x}$, then the ball $B(x, r)$ contains at least one basic interval $f_k([0, 1])$ and it does not contain $[0,1]$. In this case we conclude   by Lemmas \ref{lem:upper-density-1} and  Proposition \ref{prop:upper-density}   that
\begin{equation}\label{eq:mar-4-1}
\frac{\mu(B(x, r))}{(2r)^s}\le \max\set{\frac{\rho^s}{2^s(M(k,x))^s}, \frac{(1+2\hat k)\rho^s}{2^s(\hat kR+M(k,x))^s}, \frac{1}{2^s(\max\set{x, 1-x})^s}},
\end{equation}
where the equality is attainable, and $M(k,x):=\max\set{kR+\rho-x, x-kR}$.

\begin{proof}[Proof of Theorem \ref{main:densities} (ii)]
The proof is similar to (i).
Take $x=\pi(d_1d_2\ldots)\in E$ and $r\in(0, \rho)$. Then there exists $n\ge 0$ such that $B(x, r)$ contains the level-$(n+1)$ basic interval $f_{d_1\ldots d_{n+1}}([0, 1])$ but it does not contain the basic interval $f_{d_1\ldots d_{n}}([0,1])$. This implies that
\[
(-1, 2)\supseteq T^{n}(B(x,r) )=(f_{d_1\ldots d_{n}})^{-1}(B(x,r) )\supseteq f_{d_{n+1}}([0,1]),
\]
but $T^{n}(B(x,r) )$ does not contain    $[0,1]$.

Let $y=T^{n}x=\pi(d_{n+1}d_{n+2}\ldots)$ and let $r'=\rho^{-n}r$. Then $y\in f_{d_{n+1}}([0, 1])$, $\rho/2<r'<1$ and $T^{n}(B(x,r))=B(y,r')$. By (\ref{eq:mar-4-1}) it follows that
\begin{equation}\label{eq:mar-4-2}
\begin{split}
 &\frac{\mu(B(x,r))}{(2r)^s}=\frac{\mu(B(y, r'))}{(2r')^s}\\
 &\le \max\set{\frac{\rho^s}{2^s(M(d_{n+1}, y))^s}, \frac{(1+2\hat d_{n+1})\rho^s}{2^s(\hat d_{n+1}R+M(d_{n+1}, y))^s}, \frac{1}{2^s(\max\set{y, 1-y})^s}}\\
 &=\max\left\{\frac{\rho^s}{2^s(M(d_{n+1}, T^n x))^s}, \frac{(1+2\hat d_{n+1})\rho^s}{2^s(\hat d_{n+1}R+M(d_{n+1}, T^n x))^s},\right.\\
 &\hspace{8cm}\left. \frac{1}{2^s(\max\set{T^n x, 1-T^n x})^s}\right\}.
\end{split}
\end{equation}
Observe that
$
T^{n}x-d_{n+1} R= \rho T^{n+1}(x).
$
Then
\begin{align*}
M(d_{n+1}, T^n x)&=\max\set{ T^nx-d_{n+1}R, d_{n+1}R+\rho-T^n x}\\
&=\rho\max\set{T^{n+1} x, 1-T^{n+1} x}=\rho\eta_{n+1}(x).
\end{align*}
Substituting this in (\ref{eq:mar-4-2}) gives that
\[
\frac{\mu(B(x,r))}{(2r)^s}\le \max\set{\frac{1}{2^s(\eta_{n+1}(x))^s}, \frac{ 1+2\hat d_{n+1} }{2^s(\hat d_{n+1}R/\rho+\eta_{n+1}(x))^s}, \frac{1}{2^s(\eta_n(x))^s}}.
\]
Letting $r\to 0$ in the above equation, and then $n\to\f$, we obtain  that
\begin{equation}\label{eq:mar-4-3}
\begin{split}
\Theta^{*s}(\mu, x)&=\limsup_{r\to 0}\frac{\mu(B(x,r))}{(2r)^s}\\
&\le  \limsup_{n\to\f}\max\set{\frac{1}{2^s(\eta_{n+1}(x))^s}, \frac{ 1+2\hat d_{n+1} }{2^s(\hat d_{n+1}R/\rho+\eta_{n+1}(x))^s}, \frac{1}{2^s(\eta_n(x))^s}}\\
&= \max\set{\frac{1}{2^s(\liminf_{n\to\f}\eta_n(x))^s}, \limsup_{n\to\f}\frac{1+2\hat d_n}{2^s(\hat d_n R/\rho+\eta_n(x))^s}}.
\end{split}
\end{equation}

Observe that the equality in (\ref{eq:mar-4-2}) is attainable. This implies that the equality holds in (\ref{eq:mar-4-3}),   completing  the proof.
\end{proof}

\subsection{Typical values of the densities} In the following   we will prove Theorem \ref{main:densities} (iii) for  the typical pointwise densities of $\mu$. First we need the following upper bound.
\begin{lemma}\label{lem:upper-bound-density}
If $N=2m$ for some $m\ge 1$, then
\[
\Theta^{*s}(\mu, x)\le \frac{1}{(NR)^s}
\]
for any $x\in E$.
\end{lemma}
\begin{proof}
Observe that for any $x\in E$,
\[
\eta_n(x)=\max\set{T^n x, 1-T^n x}\ge \min_{y\in E}\max\set{y, 1-y}= mR=\frac{N}{2}R.
\]
So, by Theorem \ref{main:densities} (ii) it suffices to prove
\begin{equation}
  \label{eq:apr-25-2}
  \frac{1+2 j}{(2 j R/\rho+N R)^s }\le \frac{1}{(NR)^s}\quad \textrm{for any }1\le j\le m-1.
\end{equation}
By the same way as in the proof of Lemma \ref{lem:inequality} (ii) one can verify that the sequence
\[
b_j':=\frac{1+2 j}{(2 j R/\rho+N R)^s }, \quad j\ge 1
\]
 is strictly increasing. So, to prove (\ref{eq:apr-25-2}) we only need to prove it for $j=m-1$, i.e.,
 \[
 \frac{N-1}{((N-2)R/\rho+NR)^s}\le \frac{1}{(NR)^s}.
 \]

Note that $s=-\log N/\log \rho$. Rearranging the above inequality gives
\begin{equation}\label{eq:apr-25-3}
N-1\le \left(1+\frac{N-2}{N\rho}\right)^{-\frac{\log N}{\log \rho}}=:\phi(\rho).
\end{equation}
Clearly, (\ref{eq:apr-25-3}) holds for $N=2$.  When $N\ge 3$, note that $\phi(1/N^2)=N-1$. So it suffices to prove that
  the function
$\rho\mapsto \phi(\rho)$
is strictly decreasing on $(0, 1/N^2]$. Write $t:=1/\rho$. Then $\phi(\rho)$ is strictly decreasing if and only if
\begin{equation}
  \label{eq:apr-29-1}
  \phi_1(t)=\frac{1}{\ln t}\ln\left(1+\frac{N-2}{N}t\right)\quad\textrm{ is strictly increasing on }[N^2, +\f),
\end{equation}
where `$\ln$' is the logarithm with the natural base $e$.
Taking the derivative of $\phi_1$ gives that $\phi_1'(t)>0$ if and only if
\[
\phi_2(t)=\frac{N-2}{N}t\ln t-\left(1+\frac{N-2}{N}t\right)\ln\left(1+\frac{N-2}{N}t\right)>0.
\]
Since $t\ge N^2$, one can easily verify that $\phi_2$ has positive derivative on $[N^2, +\f)$. So,
\[
\phi_2(t)\ge \phi_2(N^2)=2\left(N(N-2)\ln N-(N-1)^2\ln(N-1)\right).
\]
Therefore, to prove (\ref{eq:apr-29-1}) it suffices to prove
\[
\frac{N(N-2)}{(N-1)^2}\ge \frac{\ln(N-1)}{\ln N}.
\]
Using one minus both sides of the above inequality   it follows that
\[
\frac{1}{(N-1)^2}\le \frac{\ln(1+\frac{1}{N-1})}{\ln N}=\frac{\ln(1+\frac{1}{N-1})^N}{N\ln N}.
\]
Since $(1+\frac{1}{N-1})^N$ decreases to $e$ as $N\to\f$, we have $\ln(1+\frac{1}{N-1})^N>1$, and thus it suffices to prove
\[
N\ln N\le (N-1)^2\quad\textrm{for all }N\ge 3.
\]
This can be easily verified by simple calculation.

 Therefore, we establish (\ref{eq:apr-25-3}), and thus (\ref{eq:apr-25-2}). This completes the proof.
\end{proof}

\begin{proof}[Proof of Theorem \ref{main:densities} (iii)]
First we consider the typical value for $\Theta_*^s(\mu, x)$. By (i)    it suffices to prove
\begin{equation}\label{eq:mar-5-1}
\liminf_{n\to\f}\ga_n(x)=0\quad\textrm{for }\mu-\textrm{almost every }x\in E.
\end{equation}
For $\ell\ge 1$ let
 \begin{align*}
 A_\ell&:=\bigcap_{k=0}^\f\set{\pi(d_1d_2\ldots)\in E: d_{k\ell+1}\ldots d_{(k+1)\ell}\ne 0^\ell},\\
 \end{align*}
Then  for any $x\in E\setminus\bigcup_{\ell=1}^\f A_\ell$ its unique coding must contain arbitrarily long length of consecutive zeros.
This implies that
$
\liminf_{n\to\f}\ga_n(x)=0
$
for any $x\in E\setminus\bigcup_{\ell=1}^\f A_\ell$.

Therefore, to prove (\ref{eq:mar-5-1}) it suffices to prove that $\mu(A_\ell)=0$ for all $\ell\ge 1$. Take $\ell\ge 1$. Observe that $A_\ell$ is the self-similar set generated by the IFS
\[
\set{f_{i_1\ldots i_\ell}:~ i_1\ldots i_\ell\in\set{0,1,\ldots, N-1}^\ell\textrm{ but }i_1\ldots i_\ell\ne 0^\ell},
\] which satisfy  the open set condition. So,
\[
\dim_H A_\ell=\frac{\log(N^\ell-1)}{-\ell\log\rho}<\frac{\log N}{-\log \rho}=s=\dim_H E.
\]
This implies that $\mathcal H^s(A_\ell)=0$, and thus $\mu(A_\ell)=\mathcal H^s(E\cap A_\ell)=0$. Hence, we prove (\ref{eq:mar-5-1}).

Next we turn to   the typical value of the upper density. First we assume $N=2m+1$ for some $m\ge 1$. Note by \cite[Theroem 5.1]{Falconer_1990} that $\Theta^{*s}(\mu, x)\le 1$ for all $x\in E$. Then by Theorem \ref{main:densities} (ii) it suffices to prove
\begin{equation}
  \label{eq:apr-25-1}
  \liminf_{n\to\f}\eta_n(x)=\frac{1}{2}\quad\textrm{for }\mu-\textrm{almost every }x\in E.
\end{equation}
For $\ell\ge 1$ let
\[
B_\ell:=\bigcap_{k=0}^\f\set{\pi(d_1d_2\ldots)\in E: d_{k\ell+1}\ldots d_{(k+1)\ell}\ne m^\ell }.
\]
Then for any $x\in E\setminus\bigcup_{\ell\ge 1}B_\ell$ its unique coding must contain arbitrarily long length of consecutive $m$'s. Note that $\eta_n(x)\ge 1/2$ for any $x\in E$. Therefore,
$
\liminf_{n\to\f}\eta_n(x) =1/2
$
for any $x\in E\setminus\bigcup_{\ell\ge 1} B_\ell$. Furthermore, by a similar argument as above one can   verify that $\mu(B_\ell)=0$ for all $\ell\ge 1$. This proves (\ref{eq:apr-25-1}).

Finally, we consider the typical value of the upper density with $N=2m$ for some   $m\ge 1$.
Define
\[
B_\ell':=\bigcap_{k=0}^\f\set{\pi(d_1d_2\ldots)\in E: d_{k\ell+1}\ldots d_{(k+1)\ell}=m0^{\ell-1}}
\]
for any $\ell\ge 2$. So, for any $x\in E\setminus\bigcup_{\ell=2}^\f B_\ell'$ its unique coding contains the block $m 0^\ell$ with $\ell$ arbitrarily large. Note that $\eta_n(x)\ge mR=NR/2$ for any $x\in E$.  This implies
$
\liminf_{n\to\f}\eta_n(x)=NR/2
$
for any $x\in E\setminus\bigcup_{\ell=2}^\f B_\ell'$.
  Hence, by Lemma \ref{lem:upper-bound-density} and Theorem \ref{main:densities} (ii) it follows that
\[
\Theta^{*s}(\mu, x)=\frac{1}{(NR)^s}\quad\forall x\in E\setminus\bigcup_{\ell=2}^\f B_\ell'.
\]
 By the same argument as above we can show that $\mu(B_\ell')=0$ for any $\ell\ge 2$. This completes the proof.
\end{proof}
As a corollary of Theorem \ref{main:densities} (iii) we have the precise packing measure of $E$ ( see \cite[Theorem 1.2]{Feng-Hua-Wen-2000} for a proof).
\begin{corollary}\label{cor:packing-measure}
Let $N\ge 2$ and $0<\rho\le 1/N^2$. Then the $s$-dimensional  packing measure of $E$ is given by
\[
\mathcal P^s(E)=\left(\frac{2R}{\rho}\right)^s.
\]
\end{corollary}

\section{Critical values for the densities}\label{sec:critical-values}
In this section we will determine the critical values   of  the sets
\[
E_*(a)=\set{x\in E: \Theta_*^s(\mu, x)\ge a}\quad\textrm{and}\quad E^*(b)=\set{x\in E: \Theta^{*s}(\mu, x)\le b}
\]
respectively, and prove Theorem \ref{main:critical-values}. Recall that the critical values  for $E_*(a)$ and $E^*(b)$ are defined by
\[
a_c:=\sup\set{a: \dim_H E_*(a)>0}\quad\textrm{and}\quad b_c:=\inf\set{b: \dim_H E^*(b)>0}.
\]
Then by (\ref{eq:bounds}) it follows  that $a_c<b_c$.
Motivated  by some recent works on critical values of unique beta expansions (cf.~\cite{Glendinning_Sidorov_2001, Kong_Li_Dekking_2010}) and open dynamical systems \cite{Kalle-Kong-Langeveld-Li-18}, we show that the critical values  $a_c$ and $b_c$ are related to the Thue-Morse type sequences $(\la_i)$ and $(\theta_i)$ defined in Definition \ref{def:sequences}.

\subsection{Critical value of $E_*(a)$}
For $t\in[0, 1]$ let
\[\al(t)=\al_1(t)\al_2(t)\ldots:=\pi^{-1}(t'),\] where $t'$ is the smallest element of $E$ no less than $t$. So, if $t\in E$, then $\al(t)$ is indeed the unique coding of $t$. Since $E$ is a Cantor set, its complement $[0, 1]\setminus E$ is a countable union of open intervals. The definition of $\al(t)$ implies that in each connected component of $[0, 1]\setminus E$ the map $t\mapsto \al(t)$ is constant. Observe by Theorem \ref{main:densities} (i) that $\Theta_*^s(\mu, x)$ is uniquely determined by $\ga(x):=\liminf_{n\to\f}\ga_n(x)$. So it suffices to consider the critical value  of the set
\[
E_\ga(t):=\set{x\in E: \ga(x)\ge t}.
\]
To describe the set $E_\ga(t)$ it is convenient to study the corresponding  set in the coding (sequence) space.  For this reason we first recall some terminology from the symbolic dynamics (cf.~\cite{Lind_Marcus_1995}).

For a sequence $(c_i)=c_1c_2\ldots\in\set{0,1,\ldots, N-1}^\N$ we mean an infinite string of digits. Similarly,  for a word $\mathbf c=c_1\ldots c_n$ with $n\in\N$ we mean  a finite string of digits with each digit $c_i$ from $\set{0,1,\ldots, N-1}$. For two words $\mathbf c$ and $\mathbf d$ we denote by $\mathbf c\mathbf d$ the new word which is the concatenation of them. Also, for $n\in\N$ we write for $\mathbf c^n=\mathbf c\cdots\mathbf c$ the $n$ times concatenation of $\mathbf c$, and   write for $\mathbf c^\f$ the periodic sequence with period block $\mathbf c$. In this paper we use the lexicographical ordering `$\prec, \lle, \succ$' and '$\lge$' between sequences and words in the usual way. For example, for two sequences $(c_i), (d_i)\in\set{0,1,\ldots, N-1}^\N$ we say $(c_i)\prec (d_i)$ if $c_1<d_1$, or there exists $n\in\N$ such that $c_1\ldots c_n=d_1\ldots d_n$ and $c_{n+1}<d_{n+1}$. Also, we write $(c_i)\lle (d_i)$ if $(c_i)=(d_i)$ or $(c_i)\prec (d_i)$. For two words $\mathbf c$ and $\mathbf d$ not necessarily of the same length, we say $\mathbf c\prec \mathbf d$ if $\mathbf c 0^\f\prec \mathbf d 0^\f$. Recall that  for a word $\mathbf c=c_1\ldots c_n$ its reflection is defined by $\overline{\mathbf c}=(N-1-c_1)\ldots (N-1-c_n)$. If $c_n<N-1$, then we write $\mathbf c^+=c_1\ldots c_{n-1}(c_n+1)$; and if $c_n>0$ then we write $\mathbf c^-=c_1\ldots c_{n-1}(c_n-1)$. So, $\overline{\mathbf c}, \mathbf c^+$ and $\mathbf c^-$ are all words with each digit from $\set{0,1,\ldots, N-1}$. Analogously, for a sequence $(c_i)\in\set{0,1,\ldots, N-1}^\N$ we denote its \emph{reflection} by $\overline{(c_i)}:=(N-1-c_1)(N-1-c_2)\ldots$.

 Now we define the   symbolic analogue of $E_\ga(t)$.
For $t\in[0,1]$ let  $E_\ga'(t)$ be the set   of   sequences $(d_i)\in\set{0,1,\ldots, N-1}^\N$ satisfying
\[\left\{
\begin{array}{lll}
d_{n+1}d_{n+2}\ldots\lge \al(t)&\textrm{if}& d_n=0,\\
d_{n+1}d_{n+2}\ldots \lge \al(t)\textrm{ or }d_{n+1}d_{n+2}\ldots \lle \overline{\al(t)}&\textrm{if}& 1\le d_n\le N-2,\\
d_{n+1}d_{n+2}\ldots \lle \overline{\al(t)}&\textrm{if}& d_n=N-1.
\end{array}\right.\]

In the following proposition we show that $E_\ga(t)$ and $\pi(E_\ga'(t))$ have the same Hausdorff dimension for each $t\notin E$.
\begin{proposition}\label{prop:dimension}
Let $N\ge 2$ and $0<\rho\le 1/N^2$.
If $t\in[0, 1]\setminus E$, then
\[
\dim_H E_\ga(t)=\dim_H\pi(E_\ga'(t)).
\]
\end{proposition}
\begin{proof}
Let $t\in[0, 1]\setminus E$. Since $E$ is compact and $\set{0, 1}\subset E$, there exists $\ep>0$ such that $[t-\ep, t+\ep]\cap E=\emptyset$.  Take $x\in E_\ga(t)$. Then  $\ga(x)=\liminf_{n\to\f}\ga_n(x)\ge t$. So there exists a large integer $M$ such that
\begin{equation}\label{eq:apr-8-1}
\ga_n(x)\ge t-\ep\quad\forall ~n\ge M.
\end{equation}
Write $x=\pi(d_1d_2\ldots)$. We claim that
$
d_{M+1}d_{M+2}\ldots \in E_\ga'(t).
$
Take $n\ge M+1$. We will prove the claim by considering the following three cases.

Case I. If $d_n=0$, then by (\ref{eq:apr-8-1}) it follows that
\begin{equation}\label{eq:apr-8-2}
 \pi(d_{n+1}d_{n+2}\ldots )=T^n x=\ga_n(x)\ge t-\ep.
\end{equation}
Since $[t-\ep, t+\ep]\cap E=\emptyset$, by (\ref{eq:apr-8-2}) and the definition of $\al(t)$ it follows that
\begin{equation}\label{eq:apr-28-1}
\pi(d_{n+1}d_{n+2}\ldots )\ge \pi(\al(t-\ep))= \pi(\al(t)).
\end{equation}
Observe that the projection map $\pi: \set{0,1\ldots, N-1}^\N\to E$ is bijective and strictly increasing with respect to the lexicographical ordering in $\set{0,1,\ldots, N-1}^\N$. We then conclude from (\ref{eq:apr-28-1}) that
\[
d_{n+1}d_{n+2}\ldots \lge \al(t)
\]
as desired.

Case II. If $d_n\in\set{1,\ldots, N-2}$, then again by (\ref{eq:apr-8-1}) we obtain that
\[
 \max\set{\pi(d_{n+1}d_{n+2}\ldots ), \pi(\overline{d_{n+1}d_{n+2}\ldots})}=\max\set{T^n x, 1-T^n x}=\ga_n(x)\ge t-\ep.
\]
By the same argument as in Case I we   deduce that
\[
d_{n+1}d_{n+2}\ldots \lge \al(t)\quad\textrm{or}\quad \overline{d_{n+1}d_{n+2}\ldots }\lge \al(t).
\]

Case III. If $d_n=N-1$, then by (\ref{eq:apr-8-1}) we also have
\[
\pi(\overline{d_{n+1}d_{n+2}\ldots})=1-T^n x=\ga_n(x)\ge t-\ep.
\]
Using the same argument as in Case I we obtain that $\overline{d_{n+1}d_{n+2}\ldots}\lge \al(t)$. This establishes the claim.

Therefore, by the claim it follows that
\[
E_\ga(t)\subset\bigcup_{k=0}^\f\bigcup_{\mathbf i\in\set{0,1,\ldots, N-1}^k}f_{\mathbf i}(\pi(E_\ga'(t))),
\]
which gives $\dim_H E_\ga(t)\le \dim_H\pi(E_\ga'(t))$.

On the other hand,   take $x=\pi(d_1d_2\ldots )\in\pi(E_\ga'(t))$. Then by the same argument as above we can deduce that
\[
\ga _n(x)\ge \pi(\al(t))\ge t\quad\forall n\ge 0.
\]
This implies that $\ga(x)=\liminf_{n\to\f}\ga_n(x)\ge t$. So, $\pi(E_\ga'(t))\subset E_\ga(t)$, and thus $\dim_H E_\ga(t)=\dim_H\pi(E_\ga'(t))$.
\end{proof}

Recall from Definition \ref{def:sequences} that  $(\la_i)_{i=1}^\f$ is the Thue-Morse type sequence  satisfying
\[
\la_1=N-1,\quad\textrm{and}\quad \la_{2^n+1}\ldots \la_{2^{n+1}}=\overline{\la_1\ldots \la_{2^n}}^+\quad\forall n\ge 0.
\]
Then $(\la_i)$ begins with
$
 (N-1)10(N-1)0(N-2)(N-1)10\ldots.
 $
\begin{lemma}\label{lem:prop-lambda}\mbox{}

\begin{enumerate}[{\rm(i)}]
\item For any $n\in\N$ we have
\begin{equation}\label{eq:inequality-lambda}
\overline{\la_1\ldots \la_{2^n-i}}\prec \la_{i+1}\ldots \la_{2^n}\lle \la_1\ldots\la_{2^n-i}\quad \forall ~ 0\le i<2^n.
\end{equation}
\item If $\la_i\in\set{1,\ldots, N-2}$ for some $i\ge 1$, then $\la_{i+1}\in\set{0, N-1}$.

\end{enumerate}
\end{lemma}
\begin{proof}
First we prove (i). We will  prove (\ref{eq:inequality-lambda}) by induction on $n$. Note that  $\la_1\la_2=(N-1)1$. Since $N\ge 2$, it is clear that (\ref{eq:inequality-lambda}) holds for $n=1$. Now suppose (\ref{eq:inequality-lambda}) holds for some $n\ge 1$, and we    will prove (\ref{eq:inequality-lambda}) with $n$ replaced by $n+1$, i.e.,
\begin{equation}
  \label{eq:apr-28-2}
  \overline{\la_1\ldots\la_{2^{n+1}-i}}\prec \la_{i+1}\ldots \la_{2^{n+1}}\lle \la_1\ldots \la_{2^{n+1}-i}\quad\forall ~0\le i<2^{n+1}.
\end{equation}
Clearly, (\ref{eq:apr-28-2}) holds for $i=0$ since $\la_1>\overline{\la_1}$. So it suffices to prove (\ref{eq:apr-28-2}) for $0<i<2^{n+1}$. We consider the following two cases.

Case I. $0< i<2^n$. Then by the induction hypothesis it follows that
\begin{equation}\label{eq:apr-7-11}
\overline{\la_1\ldots \la_{2^n-i}}\prec \la_{i+1}\ldots \la_{2^n}\lle \la_1\ldots\la_{2^n-i}
\quad\textrm{and}\quad
\overline{\la_1\ldots \la_i}\prec \la_{2^n-i+1}\ldots \la_{2^n}.
\end{equation}
Since $\la_{2^n+1}\ldots \la_{2^n+i}=\overline{\la_1\ldots \la_i}$, we then obtain by (\ref{eq:apr-7-11})    that
\begin{equation*}
\overline{\la_1\ldots\la_{2^n}}\prec \la_{i+1}\ldots \la_{2^n+i}\prec \la_1\ldots \la_{2^n}\quad \textrm{for } 1\le i< 2^n.
\end{equation*}
This proves (\ref{eq:apr-28-2}) for $0< i<2^n$.

Case II. $2^n\le i<2^{n+1}$. Write $i':=i-2^n$. Then by Definition \ref{def:sequences} it follows that
\begin{equation}\label{eq:apr-7-13}
\la_{i+1}\ldots \la_{2^{n+1}}=\overline{\la_{i'+1}\ldots \la_{2^n}}^+=\overline{\la_{i'+1}\ldots \la_{2^n}^-}.
\end{equation}
Note that $0\le i'<2^n$. By the induction hypothesis it follows that
\[
\overline{\la_1\ldots \la_{2^n-i'}}\lle \la_{i'+1}\ldots \la_{2^n}^-\prec \la_1\ldots \la_{2^n-i'}.
\]
Taking the reflection on both sides, and then by (\ref{eq:apr-7-13}) it follows that
\[
\overline{\la_1\ldots \la_{2^{n+1}-i}}\prec \la_{i+1}\ldots \la_{2^{n+1}}\lle \la_1\ldots \la_{2^{n+1}-i}.
\]
This proves (\ref{eq:apr-28-2}) for $2^n\le i<2^{n+1}$.

Therefore, by Cases I and II we establish  (\ref{eq:apr-28-2}). This completes the proof of (i) by induction.

 Now we turn to prove (ii). We will   prove by induction on $n$ that
 \begin{equation}\label{eq:apr-7-21}
 \la_i\in\set{1,\ldots, N-2}\quad\textrm{for some }1\le i<2^n\quad\Longrightarrow\quad \la_{i+1}\in\set{0, N-1}.
 \end{equation}
Note that  (\ref{eq:apr-7-21}) is trivial for $n=1$, since $\la_1\la_2=(N-1)1$. Now suppose (\ref{eq:apr-7-21}) holds for some $n\ge 1$, and we consider it for $n+1$. Suppose $\la_{i}\in\set{1,\ldots, N-2}$ for some $1\le i<2^{n+1}$.  We consider the following four cases.

 Case 1. If $1\le i<2^n$, then by the  induction hypothesis it follows that $\la_{i+1}\in\set{0, N-1}$.

 Case 2. If $i=2^n$, then by Definition \ref{def:sequences} we have $\la_{i+1}=\overline{\la_1}=0$.

 Case 3. If $2^n< i<2^{n+1}-1$, then write $i':=i-2^n$, and by the definition of $(\la_i)$ it follows that $\la_{i'}=\overline{\la_i}\in\set{1,\ldots, N-2}$. So by the induction hypothesis we have $\la_{i'+1}\in\set{0, N-1}$. This implies $\la_{i+1}=\overline{\la_{i'+1}}\in\set{0, N-1}$.

 Case 4. If $i=2^{n+1}-1$, then $\la_{2^{n+1}-1}=\overline{\la_{2^n-1}}\in\set{1,\ldots, N-2}$. This gives $\la_{2^n-1}\in\set{1,\ldots, N-2}$. Note that $\la_{2^n-1}=\overline{\la_{2^{n-1}-1}}$. So we also have $\la_{2^{n-1}-1}\in\set{1,\ldots, N-2}$. Proceeding this argument we can deduce that $\la_1\in\set{1,\ldots, N-2}$, leading to a contradiction with $\la_1=N-1$. So, $\la_{2^{n+1}-1}\notin\set{1,\ldots, N-2}$.

 By Cases 1--4 we prove (\ref{eq:apr-7-21}) with $n$ replaced by $n+1$. By induction this proves (ii).
 \end{proof}

Now we are ready to determine the critical values of $E_\ga(t)$, and show that it is equal to
\[t_\ga:=\pi(\overline{\la_2\la_3\ldots})=1-R\sum_{i=1}^\f\la_{i+1}\rho^{i-1}.\]
\begin{lemma}\label{lem:lower-bound-gamma-2}
If $t<t_\ga$, then $\dim_H E_\ga(t)>0$.
\end{lemma}
\begin{proof}
Let $s_n=\pi(\overline{\la_2\ldots \la_{2^n+1}}0^\f)$. Then $s_n\nearrow t_\ga$ as $n\to\f$. Note that the set-valued map $t\mapsto E_\ga(t)$ is non-increasing and $E$ is a Cantor set. So by Proposition \ref{prop:dimension} it suffices to prove $\dim_H \pi(E_\ga'(s_n))>0$ for all $n\ge 1$. Fix $n\ge 1$ and write
\[\xi:=\overline{\la_2\ldots \la_{2^n+1}}\quad \textrm{and}\quad \zeta:=\overline{\la_{2^n+2}\ldots \la_{2^{n+1}+1}}.\]
 We claim that the words
$
\xi\zeta, \xi\overline{\xi}$ and $\zeta\overline{\xi}
$
are all
 admissible in $E_\ga'(s_n)$.

First we show that $\xi\zeta$ is admissible in $E_\ga'(s_n)$. Note that
\begin{equation}\label{eq:apr-9-1}
\xi\zeta=\overline{\la_2\ldots \la_{2^{n+1}+1}}=\overline{\la_2\ldots \la_{2^n}}\la_1\ldots \la_{2^n}^-\la_1=:c_1\ldots c_{2^{n+1}}.
\end{equation}
We will show for all $1\le i\le 2^n$ that
\begin{equation}\label{eq:apr-4-2}
\left\{
\begin{array}{lll}
c_{i+1}\ldots c_{i+2^{n}}\lge\overline{\la_2\ldots \la_{2^{n}+1}}&\textrm{if}& c_i=0,\\
c_{i+1}\ldots c_{i+2^{n}}\lge\overline{\la_2\ldots \la_{2^{n}+1}}\textrm{ or }c_{i+1}\ldots c_{i+2^{n}}\lle \la_2\ldots \la_{2^n+1}&\textrm{if}&1\le c_i\le N-2,\\
c_{i+1}\ldots c_{i+2^{n}}\lle \la_2\ldots \la_{2^n+1}&\textrm{if}& c_i=N-1.
\end{array}\right.
\end{equation}
Observe  that if $i=2^n$, then by (\ref{eq:apr-9-1}) we have $c_{2^n}=\la_1=N-1$, and  thus
\[
c_{i+1}\ldots c_{i+2^n}=\la_2\ldots \la_{2^n}^-\la_1\prec \la_2\ldots\la_{2^n+1}
\]
as desired.
In the following it suffices to prove
 (\ref{eq:apr-4-2}) for $1\le i<2^n$. We consider the following three  cases.

Case I. $c_i=0$ for some $1\le i< 2^n$. Then by (\ref{eq:inequality-lambda}) and (\ref{eq:apr-9-1}) it follows that
\[
c_{i}\ldots c_{i+2^n-1}=\overline{\la_{i+1}\ldots \la_{2^n}}\la_1\ldots \la_i\succ \overline{\la_1\ldots \la_{2^n}},
\]
which together with $\overline{\la_1}=0=c_i$ implies
\[c_{i+1}\ldots c_{i+2^n-1}\succ \overline{\la_2\ldots \la_{2^n}}\]
as required.

Case II.  $c_i=N-1$ for some $1\le i<2^n$. Note by (\ref{eq:inequality-lambda}) and (\ref{eq:apr-9-1}) that $c_{2^{n}-1}=\overline{\la_{2^n}}<N-1$. So  we  have $1\le i<2^n-1$. By (\ref{eq:inequality-lambda}) and (\ref{eq:apr-9-1}) it follows that
\[
c_i\ldots c_{2^n-1}=\overline{\la_{i+1}\ldots \la_{2^n}}\prec \la_1\ldots \la_{2^n-i}.
\]
This, together with  $\la_1=N-1=c_i$,  implies $c_{i+1}\ldots c_{2^n-1}\prec \la_2\ldots \la_{2^n-i}$.

Case III. $1\le c_i\le N-2$ for some $1\le i< 2^n$. Then $N\ge 3$, and by Lemma \ref{lem:prop-lambda} (ii) we have $c_{i+1}\in\set{0, N-1}$. Since $\overline{\la_2}=N-2$, it gives that either $c_{i+1}>\overline{\la_2}$ or $c_{i+1}<\la_2$.

By Cases I--III we establish (\ref{eq:apr-4-2}). Similarly,  note  that
\[
\xi\overline{\xi}=\overline{\la_2\ldots \la_{2^{n}+1}}\la_2\ldots \la_{2^n+1}=\overline{\la_2\ldots \la_{2^n}}\la_1\ldots \la_{2^n+1}=:c_1'\ldots c_{2^{n+1}}'
\] and
\[\zeta\overline{\xi}=\overline{\la_{2^n+2}\ldots \la_{2^{n+1}+1}}\la_2\ldots \la_{2^n+1}=\la_2\ldots \la_{2^n}^-\la_1\ldots \la_{2^n+1}=:c_1''\ldots c_{2^{n+1}}''.
\]
Then by  using Lemma \ref{lem:prop-lambda} and the same argument as above we can prove the analogous inequalities in (\ref{eq:apr-4-2}) with   $c_1\ldots c_{2^{n+1}}$ replaced by $c_1'\ldots c_{2^{n+1}}'$ and $c_1''\ldots c_{2^{n+1}}''$, respectively.

 Therefore, by using $\al(s_n)=\overline{\la_2\ldots \la_{2^n+1}}0^\f$  it follows that the words $\xi\zeta, \xi\overline{\xi}$ and $\zeta\overline{\xi}$ are all admissible in $E'_\ga(s_n)$, proving the claim.  Observe  that the set $E_\ga'(t)$ is symmetric, i.e.,  $(d_i)\in E_\ga'(t)$ if and only if $\overline{(d_i)}\in E_\ga'(t)$. By the claim it follows that the words
 \[
 \xi\zeta,\quad \overline{\xi\zeta},\quad\xi\overline{\xi},\quad\overline{\xi}\xi,\quad \zeta\overline{\xi}\quad\textrm{and}\quad \overline{\zeta}\xi
 \]
 are all admissible in $E_\ga'(s_n)$, and hence $E_\ga'(s_n)$ contains the  subshift of finite type $X_A$ over the states $\set{\xi, \zeta, \overline{\xi}, \overline{\zeta}}$ with the adjacency matrix
\[
A=\left(
\begin{array}{llll}
0&1&1&0\\
0&0&1&0\\
1&0&0&1\\
1&0&0&0
\end{array}\right).
\]
This implies that
\[
\dim_H \pi(E'_\ga(s_n))\ge \dim_H\pi(X_A)=\frac{\log r_A}{-\log \rho^{2^n}}=\frac{\log \frac{1+\sqrt{5}}{2}}{-2^n\log \rho}>0,
\]
where $r_A=(1+\sqrt{5})/2$ is the spectral radius of $A$. This completes
  the proof.
\end{proof}

\begin{lemma}
\label{lem:upper-bound-gamma-2}
If $t>t_\ga$, then  the set  $E_\ga(t)$ is countable.
\end{lemma}
\begin{proof}
Let $t_n=\pi(\overline{\la_2\ldots \la_{2^n+1}}(N-1)^\f)$. Then $t_n\searrow t_\ga$ as $n\to\f$. Note that the set-valued map $t\mapsto E_\ga(t)$ is non-increasing, and $E_\ga(t)$ is a countable union of scaling copies of $\pi(E_\ga'(t))$. So it suffices to prove $E_\ga'(t_n)$ is countable for all $n\ge 1$. Observe that for $n=1$ we have $\al(t_1)=(N-2)(N-1)^\f$. One can easily verify that any sequence in $E_\ga'(t_1)$ must end with $(0(N-1))^\f$, and thus $E_\ga'(t_1)$ is countable. Since $E_\ga'(t_{n+1})\supseteq E_\ga'(t_n)$ for any $n\ge 1$, in the following we only need  to prove that $E_\ga'(t_{n+1})\setminus E_\ga'(t_n)$ is countable for all $n\ge 1$.

Fix $n\ge 1$. Note that $\al(t_n)$ and $\al(t_{n+1})$ both begin with $\overline{\la_2\la_3}=(N-2)(N-1)$. Then
\begin{equation}\label{eq:apr-28-3}
\begin{array}{lll}
\overline{\al(t_n)}\prec \overline{\al(t_{n+1})}\prec  \al(t_{n+1})\prec \al(t_n)&\quad\textrm{if}& N\ge 3,\\
\al(t_{n+1})\prec \al(t_n)\prec \overline{\al(t_n)}\prec \overline{\al(t_{n+1})}&\quad\textrm{if}& N=2.
\end{array}
\end{equation}
Take $(d_i)\in E_\ga'(t_{n+1})\setminus E_\ga'(t_n)$. Then by (\ref{eq:apr-28-3}) and the definition of $E_\ga'(t)$ there must exist  $k\in\N$ such that
\begin{equation}\label{eq:apr-4-4}
d_k\le N-2\quad\textrm{and}\quad \al(t_{n+1})\lle d_{k+1}d_{k+2}\ldots \prec \al(t_n)
\end{equation}
or
\begin{equation}\label{eq:apr-4-5}
d_k\ge 1\quad\textrm{and}\quad \overline{\al(t_{n})}\prec d_{k+1}d_{k+2}\ldots \lle \overline{\al(t_{n+1})}.
\end{equation}
Write $v_n:=\la_2\ldots \la_{2^n+1}$. We will show in each case that the sequence $(d_i)$ must end with $(v_n\overline{v_n})^\f$.

Suppose (\ref{eq:apr-4-4}) holds.   Then $d_k\le N-2$, and by using $\al(t_n)=\overline{\la_2\ldots \la_{2^n+1}}(N-1)^\f$ it follows that
\begin{equation}\label{eq:apr-4-6}
\begin{split}
&d_{k+1}d_{k+2}\ldots\prec \overline{\la_2\ldots \la_{2^n+1}}(N-1)^\f, \\
&d_{k+1}d_{k+2}\ldots\lge\overline{\la_2\ldots \la_2^{n+1}+1}  (N-1)^\f= \overline{\la_2\ldots \la_{2^n+1}}\la_2\ldots \la_{2^n}^- (N-1)^\f.
\end{split}
\end{equation}
This implies $d_{k+1}\ldots d_{k+2^n}=\overline{\la_2\ldots \la_{2^n+1}}=\overline{v_n}$. In particular, $d_{k+2^n}=\overline{\la_{2^n+1}}=\la_1=N-1$. So, by using $(d_i)\in E_\ga'(t_{n+1})$ it follows that
\begin{equation}\label{eq:apr-4-7}
d_{k+2^n+1}\ldots d_{k+2^{n+1}}\lle \la_2\ldots \la_{2^n+1}.
\end{equation}
This together with (\ref{eq:apr-4-6}) gives $d_{k+2^n+1}\ldots d_{k+2^{n+1}-2}=\la_2\ldots \la_{2^n-1}$. If $d_{k+2^{n+1}-1}=\la_{2^n}^-$, then (\ref{eq:apr-4-6}) implies that $(d_i)$ must end with $(N-1)^\f$, leading to a contradiction with $(d_i)\in E_\ga'(t_{n+1})$. So, $d_{k+2^{n+1}-1}=\la_{2^n}.$ By (\ref{eq:apr-4-7}) and using $\la_{2^n+1}=\overline{\la_1}=0$ it follows that $d_{k+2^n+1}\ldots d_{k+2^{n+1}}=\la_2\ldots \la_{2^n+1}$. Hence,
\begin{equation}\label{eq:apr-4-8}
d_{k+1}\ldots d_{k+2^{n+1}}=\overline{\la_2\ldots \la_{2^n+1}}\la_2\ldots \la_{2^n+1}=\overline{v_n}v_n.
\end{equation}

Now, since $d_{k+2^n}=\overline{\la_{2^{n}+1}}=N-1$ and $(d_i)\in E_\ga'(t_{n+1})$, by (\ref{eq:apr-4-8}) it follows that
\begin{equation}\label{eq:apr-4-9}
d_{k+2^{n+1}+1}\ldots d_{k+2^{n+1}+2^n}\lle \la_{2^n+2}\ldots \la_{2^{n+1}+1}=\overline{\la_2\ldots \la_{2^n}}^+ 0.
\end{equation}
Also, by (\ref{eq:apr-4-8}) that $d_{k+2^{n+1}}=\la_{2^n+1}=0$  it follows that
\begin{equation}\label{eq:apr-4-10}
d_{k+2^{n+1}+1}\ldots d_{k+2^{n+1}+2^n}\lge \overline{\la_2\ldots \la_{2^n+1}}=\overline{\la_1\ldots \la_{2^n}}(N-1).
\end{equation}
By (\ref{eq:apr-4-9}) and (\ref{eq:apr-4-10}) it follows that
\[
\textrm{either }d_{k+2^{n+1}+1}\ldots d_{k+2^{n+1}+2^n}=\overline{\la_2\ldots \la_{2^n+1}}\quad\textrm{or}\quad d_{k+2^{n+1}+1}\ldots d_{k+2^{n+1}+2^n}=\overline{\la_2\ldots \la_{2^n}}^+ 0.
\]
While in the second case we have by (\ref{eq:apr-4-8}) and (\ref{eq:apr-4-9}) that
\[d_{k+2^n}=N-1\quad\textrm{and}\quad d_{k+2^n+1}\ldots d_{k+2^{n+1}+2^n}=\la_2\ldots \la_{2^{n+1}+1}.\] Using $(d_i)\in E_\ga'(t_{n+1})$ this implies $(d_i)$ must end  with $0^\f$, leading to a contradiction. So, we  have
$d_{k+2^{n+1}+1}\ldots d_{k+2^{n+1}+2^n}=\overline{\la_2\ldots \la_{2^n+1}}=\overline{v_n}$.
Proceeding this reasoning we conclude that  $(d_i)$  must  end with $(\overline{v_n}v_n)^\f$.

 Symmetrically,  if (\ref{eq:apr-4-5}) holds, then one can also show that $(d_i)$ ends with $(\overline{v_n}v_n)^\f$. This completes the proof.
\end{proof}
\begin{proof}
[Proof of Theroem \ref{main:critical-values} (A)]
Observe by Theorem \ref{main:densities} that
\[
E_\ga(t)=\set{x\in E: \ga(x)\ge t}=\set{x\in E: \Theta_*^s(\mu, x)\ge \frac{1}{2^s(R/\rho-t)^{s}}}=E_*\left(\frac{1}{2^s(R/\rho-t)^{s}}\right).
\]
So it suffices to prove that $t_\ga$ is the critical values of $E_\ga(t)$.
By  Lemmas \ref{lem:lower-bound-gamma-2} and \ref{lem:upper-bound-gamma-2} we only need to consider $E_\ga(t)$  for $t=t_\ga$. By the proof of Lemma \ref{lem:upper-bound-gamma-2} it follows that  the set $E_\ga'(t_\ga)$ contains   all sequences of the form
\[
 ((N-1)0)^{k_0}(v_1\overline{v_1})^{k_1}\cdots(v_n\overline{v_n})^{k_n}\cdots,\quad   k_n\in\set{0,1,2,\ldots}\cup\set{\f},
\]and their reflections,
where $v_n=\overline{\la_2\ldots \la_{2^n+1}}$. This implies that $\pi(E'_\ga(t_\ga))$ is uncountable. Since $\pi(E'_\ga(t_\ga))\subset E_\ga(t_\ga)$, this proves the uncountability of $E_\ga(t_\ga)$.
\end{proof}

\subsection{Critical value of  $E^*(b)$} Now we turn to describe  the critical value of $E^*(b)$. Let $\eta(x):=\liminf_{n\to\f}\eta_n(x)=\liminf_{n\to\f}\max\set{T^n x, 1-T^n x}$. Define
\[
E_\eta(t):=\set{x\in E: \eta(x)\ge t},
\]
and the corresponding symbolic set
\[
E_\eta'(t):=\set{(d_i): d_{n+1}d_{n+2}\ldots\lge \al(t)\quad\textrm{or}\quad d_{n+1}d_{n+2}\ldots\lle \overline{\al(t)}~\forall n\ge 0}.
\]
We  first determine the critical value of $E_\eta(t)$, and then use it to determine the critical value of $E^*(b)$.
\begin{proposition}\label{prop:dimension-1}
For any $t\in[0, 1]\setminus E$ we have
\[
\dim_H E_\eta(t)=\dim_H\pi(E_\eta'(t)).
\]
\end{proposition}
\begin{proof}
The proof is similar to Proposition \ref{prop:dimension}. Let $t\in[0, 1]\setminus E$. Choose $\ep>0$ such that $[t-\ep, t+\ep]\cap E=\emptyset$. Take $x=\pi(d_1d_2\ldots)\in E_\eta(t)$. Then $\eta(x)=\liminf_{n\to\f}\max\set{T^n x, 1-T^n x}\ge t$. So there exists a large integer $M$ such that
\[
\max\set{\pi(d_{n+1}d_{n+2}\ldots), \pi(\overline{d_{n+1}d_{n+2}\ldots})}=\max\set{T^n x, 1-T^n x}\ge t-\ep\quad \forall ~n\ge M.
\]
This implies that
\[
\max\set{\pi(d_{n+1}d_{n+2}\ldots), \pi(\overline{d_{n+1}d_{n+2}\ldots})}\ge \pi(\al(t-\ep))=\pi(\al(t))\quad\forall n\ge M.
\]
Thus,
\[
d_{n+1}d_{n+2}\ldots\lge \al(t)\quad\textrm{or}\quad d_{n+1}d_{n+2}\ldots\lle\overline{\al(t)}\quad\forall n\ge M,
\]
implying $d_{M+1}d_{M+2}\ldots \in E_\eta'(t)$. So $\dim_H E_\eta(t)\le \dim_H \pi(E_\eta'(t))$.

On the other hand, for any $x=\pi(d_1d_2\ldots)\in \pi(E_\eta'(t))$ one can verify that
\[
\max\set{T^n x, 1-T^n x}\ge \pi(\al(t))\ge t\quad\forall ~n\ge 0.
\]
This implies $\eta(x)\ge t$, and thus $x\in E_\eta(t)$. Hence, $\dim_H E_\eta(t)= \dim_H\pi(E_\eta'(t))$.
\end{proof}

Recall from Definition \ref{def:sequences} that the sequence $(\theta_i)_{i=1}^\f\in\set{0, N-1}^\N$ satisfies
\[
\theta_1=N-1,\quad\textrm{and}\quad \theta_{2^n+1}\ldots \theta_{2^{n+1}}=\overline{\theta_{1}\ldots \theta_{2^n}}\quad\forall n\ge 0.
\]
Then the sequence $(\theta_i)$ begins with
$(N-1)00(N-1)\, 0(N-1)(N-1)0\ldots.$
We will show  that the critical value of $E_\eta$ is given by
\[t_\eta:=\pi((\theta_i))=R\sum_{i=1}^\f\theta_i\rho^{i-1}.\]
First we need the following property of the sequence $(\theta_i)$.
\begin{lemma}\label{lem:property-th}
For any $n\in\N$ we have
\[
\theta_2\ldots\theta_{2^n-i+1}\lle \theta_{i+2}\ldots \theta_{2^n+1}\prec\overline{\theta_2\ldots \theta_{2^n-i+1}}\quad\forall ~ 0\le i<2^n.
\]
\end{lemma}
\begin{proof}
Recall from \cite{Allouche_Shallit_1999} that  the classical Thue-Morse sequence $(\tau_i)_{i=0}^\f$ is defined as follows. Set $\tau_0=0$, and if $\tau_0\ldots \tau_{2^n-1}$ is defined for some $n\ge 0$, then we set $\tau_{2^n}\ldots \tau_{2^{n+1}-1}=(1-\tau_0)(1-\tau_1)\ldots (1-\tau_{2^n-1})$. Comparing with  Definition \ref{def:sequences} it follows that  the sequence $(\theta_i)_{i=1}^\f$ is a variation of $(\tau_i)_{i=0}^\f$, i.e.,
\[\theta_i=(N-1)(1-\tau_{i-1})\quad \textrm{for all }i\ge 1.\]
 Therefore, the lemma follows from   the property of $(\tau_i)$ (cf.~\cite{Komornik-Loreti-1998}) that for each $n\in\N$,
\[
(1-\tau_1)\ldots(1-\tau_{2^{n}-i})\lle (1-\tau_{i+1})\ldots (1-\tau_{2^n})\prec \tau_1\ldots \tau_{2^n-i}\quad  \forall ~0\le i<2^n.
\]
\end{proof}

Now we show that $t_\eta$ is indeed  the critical value of $E_\eta(t)$.
 \begin{lemma}\label{lem:eta-lower-bound}
If $t<t_\eta$, then   $\dim_H E_\eta(t)>0$.
 \end{lemma}
 \begin{proof}
 Let $s_n:=\pi(\theta_1\ldots \theta_{2^n} 0^\f)$ with $n\in\N$. Then $s_n\nearrow  t_\eta$ as $n\to\f$. Since the map $t\mapsto E_\eta(t)$ is non-increasing, by Proposition \ref{prop:dimension-1} it suffices to prove $\dim_H \pi(E'_\eta(s_n))>0$ for any $n\in\N$. We   do this now by showing that
 \begin{equation}\label{eq:mar-31-1}
 \set{\theta_1\ldots \theta_{2^n}, \overline{\theta_1\ldots\theta_{2^n}}}^\N\subseteq E_\eta'(s_n)\quad\forall~ n\in\N.
 \end{equation}

 Fix $n\in\N$. Note that $\al(s_n)=\theta_1\ldots\theta_{2^n} 0^\f$ begins with digit $N-1$. By the definition of $E_\eta'(s_n)$ it follows that $E_\eta'(s_n)\subseteq\set{0, N-1}^\N$.
Since   $E_\eta'(s_n)$ is symmetric, to prove (\ref{eq:mar-31-1}) it suffices to prove that the words $\theta_1\ldots \theta_{2^n}\theta_1\ldots \theta_{2^n}$ and $\theta_1\ldots \theta_{2^n}\overline{\theta_1\ldots \theta_{2^n}}$ are both admissible in $E_\eta'(s_n)$. In other words, we only need to  prove for any $0\le i<2^n$ that
 \begin{equation}\label{eq:mar-31-2}
 \begin{split}
 \theta_{i+1}\ldots \theta_{2^n}\theta_1\ldots \theta_i\lle\overline{\theta_1\ldots \theta_{2^n}}\quad\textrm{or}\quad  \theta_{i+1}\ldots \theta_{2^n}\theta_1\ldots \theta_i\lge \theta_1\ldots \theta_{2^n},
 \end{split}
 \end{equation}
 and
  \begin{equation}\label{eq:mar-31-3}
 \begin{split}
 \theta_{i+1}\ldots \theta_{2^n}\overline{\theta_1\ldots \theta_i}\lle\overline{\theta_1\ldots \theta_{2^n}}\quad\textrm{or}\quad  \theta_{i+1}\ldots \theta_{2^n}\overline{\theta_1\ldots \theta_i}\lge \theta_1\ldots \theta_{2^n}.
 \end{split}
 \end{equation}
 Note that $\theta_{i+1}\in\set{0, N-1}$. We will first  prove (\ref{eq:mar-31-2}) by  considering the following two cases.

Case I.  If $\theta_{i+1}=0$, then by Lemma \ref{lem:property-th} it follows that
\begin{equation}\label{eq:apr-14-1}
\theta_{i+2}\ldots \theta_{2^n+1}\prec \overline{\theta_2\ldots \theta_{2^n-i+1}}.
\end{equation}
Note that $\theta_{i+1}=\overline{\theta_1}$. If $\theta_{i+2}\ldots \theta_{2^n}\prec \overline{\theta_2\ldots \theta_{2^n-i}}$, then  $\theta_{i+1}\ldots \theta_{2^n}\prec \overline{\theta_1\ldots \theta_{2^n-i}}$, and we are done. Otherwise, suppose $\theta_{i+2}\ldots \theta_{2^n}=\overline{\theta_2\ldots\theta_{2^n-i}}$. Then by (\ref{eq:apr-14-1}) and $(\theta_i)\in\set{0,N-1}^\N$ it follows that  $\overline{\theta_{2^n-i+1}}=N-1=\theta_1$. Again by Lemma \ref{lem:property-th} it follows that $\theta_2\ldots \theta_i\lle \overline{\theta_{2^n-i+2}\ldots \theta_{2^n}}$. So,
\[
\theta_{i+1}\ldots\theta_{2^n}\theta_1\ldots\theta_i\lle \overline{\theta_1\ldots \theta_{2^n}},
\]
proving  (\ref{eq:mar-31-2}).

Case II. If $\theta_{i+1}=N-1$, then by Lemma \ref{lem:property-th} and using $\theta_1=N-1>0=\theta_{2^n+1}$ it follows that
 \[
\theta_{i+2}\ldots \theta_{2^n}\theta_1\succ \theta_{i+2}\ldots \theta_{2^n+1}\lge \theta_2\ldots \theta_{2^n-i+1}.
\]
Since $\theta_{i+1}=\theta_1$, this again proves (\ref{eq:mar-31-2}).

Therefore, by Cases I and II we establish (\ref{eq:mar-31-2}).
Similarly, we can prove  (\ref{eq:mar-31-3}). If $\theta_{i+1}=0$, then by Lemma \ref{lem:property-th} it follows that
\[\theta_{i+2}\ldots \theta_{2^n}\overline{\theta_1}=\theta_{i+2}\ldots \theta_{2^n+1}\prec \overline{\theta_2\ldots \theta_{2^n-i+1}},\] proving (\ref{eq:mar-31-3}).
If $\theta_{i+1}=N-1$, then by Lemma \ref{lem:property-th} we can deduce that
\[
\theta_{i+1}\ldots \theta_{2^n}\overline{\theta_1}=\theta_{i+1}\ldots\theta_{2^n+1}\lge \theta_2\ldots \theta_{2^n-i+1}\quad\textrm{and}\quad
\overline{\theta_2\ldots \theta_i}\lge \theta_{2^n-i+2}\ldots \theta_{2^n}.\]
This again proves (\ref{eq:mar-31-3}).
  \end{proof}
\begin{lemma}
\label{lem:eta-upper-bound}
If $t>t_\eta$, then  the set $E_\eta(t)$ is at most countable.
\end{lemma}
\begin{proof}
Let $t_n=\pi((\theta_1\ldots \theta_{2^n})^\f)$. Note by Definition \ref{def:sequences} that
\[
\theta_1\ldots \theta_{2^{n+1}}=\theta_1\ldots \theta_{2^n}\overline{\theta_1\ldots \theta_{2^n}}^+\prec (\theta_1\ldots \theta_{2^n})^2
\]
for any $n\ge 0$. This implies  that $t_n\searrow t_\eta$ as $n\to\f$. Observe  that  the set-valued map $t\mapsto E'_\eta(t)$ is non-increasing, and  by the proof of Proposition \ref{prop:dimension-1}  that $E_\eta(t)$ is a countable union of scaling copies of $\pi(E_\eta'(t))$.  It suffices to prove that $E'_\eta(t_n)$ is countable for any $n\ge 0$. Clearly, for $n=0$ we have $t_0=\pi((N-1)^\f)$. Then one can easily verify that any sequence in $E_\eta'(t_0)$ must end with $0^\f$ or $(N-1)^\f$, and thus $E_\eta'(t_0)$ is countable. Furthermore, note that  $E_\eta'(t_{n+1})\supseteq E_\eta'(t_n)$ for all $n\ge 0$. So, it suffices to prove that $E_\eta'(t_{n+1})\setminus E_\eta'(t_n)$ is countable for all $n\ge 0$.

Fix $n\ge 0$. Then
 \[
 \overline{\al(t_n)}\prec \overline{\al(t_{n+1})}\prec \al(t_{n+1})\prec \al(t_n).
 \]Take $(d_i)\in E_\eta'(t_{n+1})\setminus E_\eta'(t_n)$. By the definition of $E_\eta'(t)$ there must exist $k\ge 0$ such that
\begin{equation}\label{eq:apr-1-1}
(\overline{\theta_1\ldots \theta_{2^n}})^\f=\overline{\al(t_n)}\prec  d_{k+1}d_{k+2}\ldots \lle \overline{\al(t_{n+1})}=(\overline{\theta_1\ldots \theta_{2^n}}\theta_1\ldots\theta_{2^n})^\f,
\end{equation}
or
\begin{equation}\label{eq:apr-1-2}
(\theta_1\ldots \theta_{2^n}\overline{\theta_1\ldots \theta_{2^n}})^\f=\al(t_{n+1})\lle  d_{k+1}d_{k+2}\ldots \prec\al(t_n)=(\theta_1\ldots\theta_{2^n})^\f.
\end{equation}
We consider the following two cases.

Case I. If (\ref{eq:apr-1-1}) holds, then there exists $m\in\N$ such that
\begin{equation}\label{eq:apr-1-3}
d_{k+1}\ldots d_{k+m 2^n}=(\overline{\theta_1\ldots \theta_{2^n}})^m\quad\textrm{and}\quad d_{k+m 2^n+1}\ldots d_{k+(m+1)2^n}\succ \overline{\theta_1\ldots \theta_{2^n}}.
\end{equation}
This, together with  $(d_i)\in E_\eta'(t_{n+1})$, implies that
\begin{equation}\label{eq:apr-1-4}
d_{k+m 2^n+1}d_{k+m 2^n+2}\ldots \lge \al(t_{n+1})=(\theta_1\ldots \theta_{2^n}\overline{\theta_1\ldots \theta_{2^n}})^\f.
\end{equation}
Note   by (\ref{eq:apr-1-3}) that  $d_{k+(m-1)2^n+1}\ldots d_{k+m 2^n}=\overline{\theta_1\ldots \theta_{2^n}}$. Then by using $(d_i)\in E_\eta'(t_{n+1})$ it gives that
\begin{equation}\label{eq:may-7-1}
d_{k+(m-1)2^n+1}d_{k+(m-1)2^n+2}\ldots\lle\overline{\al(t_{n+1})}=(\overline{\theta_1\ldots\theta_{2^n}}\theta_1\ldots \theta_{2^n})^\f.
\end{equation}
Hence, by  (\ref{eq:apr-1-3})--(\ref{eq:may-7-1}) we conclude   that
\[
d_{k+1}d_{k+2}\ldots =(\overline{\theta_1\ldots\theta_{2^n}})^m(\theta_1\ldots \theta_{2^n}\overline{\theta_1\ldots \theta_{2^n}})^\f.
\]

Case II. If (\ref{eq:apr-1-2}) holds, then by the same argument as in Case I one can show that $(d_i)$   ends with $(\theta_1\ldots\theta_{2^n}\overline{\theta_1\ldots \theta_{2^n}})^\f$.

Therefore, by Cases I and II it follows that $E_{\eta}'(t_{n+1})\setminus E_{\eta}'(t_n)$ is countable for any $n\ge 0$.
\end{proof}

\begin{proof}[Proof of Theorem \ref{main:critical-values} (B)]
Observe by Theorem \ref{main:densities} (ii) that
\begin{equation}\label{eq:apr-26-1}
E_\eta(t)=\set{x\in E: \eta(x)\ge t}\supseteq \set{x\in E: \Theta^{*s}(\mu, x)\le \frac{1}{(2t)^s}}=E^*\left(\frac{1}{(2t)^s}\right).
\end{equation}
Recall that  $t_\eta=\pi((\theta_i))$. So it suffices to prove that the critical value $b_c$ of $E^*(b)$ is given by $(2 t_\eta)^{-s}$.
First,
by Lemma  \ref{lem:eta-upper-bound} and (\ref{eq:apr-26-1}) it follows that   $E^*(b)$ is at most countable for any   $b<(2t_\eta)^{-s}$. This implies $b_c\ge (2t_\eta)^{-s}$.

Next, by Lemma  \ref{lem:eta-lower-bound} it follows that for $b>(2t_\eta)^{-s}$ we have $b^{-1/s}/2<t_\eta$, and thus
\begin{equation}\label{eq:apr-26-2}
\Lambda_n:=\pi\left(\set{\theta_1\ldots \theta_{2^n}, \overline{\theta_1\ldots \theta_{2^n}}}^\N\right)\subseteq E_\eta(b^{-1/s}/2)
\end{equation}
for sufficiently  large integer $n$. Furthermore, $\dim_H \Lambda_n>0$. Since each $x\in \Lambda_n$ has a unique coding in $\set{0, N-1}^\N$, by Theorem \ref{main:densities} (ii) and (\ref{eq:apr-26-2}) it follows that
\[
\Theta^{*s}(\mu, x)=\frac{1}{2^s\eta(x)^s}\le b\quad\forall~ x\in \Lambda_n.
\]
This  implies that
\[
\Lambda_n \subseteq\set{x\in E: \Theta^{*s}(\mu, x)\ge b}=E^*(b),
\]
and thus, $\dim_H E^*(b)>0$ for any $b>(2t_\eta)^{-s}$. This proves    $b_c=(2t_\eta)^{-s}$.

Finally we consider $E^*(b)$ for $b=b_c=(2 t_\eta)^{-s}$.
 By the proof of Lemma \ref{lem:eta-upper-bound} one can show that $E_\eta'(t_\eta)$ contains all sequences of the form
\[
 (\theta_1\overline{\theta_1})^{k_0}\cdots(\theta_1\ldots\theta_{2^n}\overline{\theta_1\ldots\theta_{2^n}})^{k_n}\cdots\quad \textrm{with}\quad k_n\in \set{0,1,\ldots}\cup\set{\f},
\]
and their reflections. Since these sequences are all in $\set{0, N-1}^\N$, by Theorem \ref{main:densities} (ii) it follows that these sequences also belong to $\pi^{-1}(E^*(b_c))$. So, $E^*(b_c)$ is uncountable, completing  the proof.
\end{proof}

\section*{Acknowlegements}
The first author was supported by  NSFC No.~11971079 and   the Fundamental
and Frontier Research Project of Chongqing No.~cstc2019jcyj-msxmX0338. The second author was supported by NSFC No.~11671147, 11571144 and Science and Technology Commission of Shanghai Municipality (STCSM) No.~13dz2260400.

\end{document}